\documentclass[12pt]{amsart}
\usepackage{amssymb}
\usepackage{amsthm}
\usepackage{amsmath,amscd,amsfonts,latexsym}
\usepackage{eucal} 
\usepackage{cmmib57} 
\usepackage{amsbsy}
\usepackage{alltt}
\usepackage{verbatim}
\usepackage{mathrsfs}
\usepackage{graphicx}
\usepackage{psfrag}

\usepackage{graphicx,enumerate}

\usepackage{geometry}
\def\ZZ         {{\mathbb Z}}
\def\RR         {{\mathbb R}}
\def\CC         {{\mathbb C}}

\def\LL         {{\mathbb L}}  
\def\QQ         {{\mathbb Q}}
\def\PP         {{\mathbb P}}

\def\Im         {{\rm Im}}

\def\dim        {{\rm dim}}

\def\deg        {{\rm deg}}
\def\codim        {{\rm codim}}

\newtheorem{theorem}{Theorem}[section]
\newtheorem{lemma}[theorem]{Lemma}
\newtheorem{prop}[theorem]{Proposition}
\newtheorem{cor}[theorem]{Corollary}
\theoremstyle{definition}
\newtheorem{dfn}[theorem]{Definition}

\newtheorem{example}[theorem]{Example}

\theoremstyle{remark}
\newtheorem{remark}[theorem]{Remark}

\begin{document}

\title{Motivic infinite cyclic covers}

\author{Manuel Gonz\'alez Villa}
\address{M. Gonz\'alez Villa: Department of Mathematics, University of Wisconsin-Madison, 480 Lincoln Drive, Madison, WI 53706, USA}
\email {villa@math.wisc.edu}

\author{Anatoly Libgober}
\address{A. Libgober: Department of Mathematics, University of Illinois at Chicago, 851 S Morgan Street, Chicago, IL 60607, USA}
\email {libgober@uic.edu}

\author{Lauren\c{t}iu Maxim}
\address{L. Maxim: Department of Mathematics, University of Wisconsin-Madison, 480 Lincoln Drive, Madison, WI 53706, USA}
\email {maxim@math.wisc.edu}


\begin{abstract}
We associate with an infinite cyclic cover of a punctured neighborhood of a simple normal 
crossing divisor on a complex quasi-projective manifold (under certain finiteness conditions)
an element in the Grothendieck ring $K_0({\rm Var}^{\hat \mu}_{\mathbb{C}})$, 
which we call {\it motivic infinite cyclic cover}, 
and show its birational invariance.
Our construction provides a unifying approach for  the Denef-Loeser motivic Milnor fibre of a complex hypersurface singularity germ, and the motivic Milnor fiber of a rational function, respectively. 
\end{abstract}

\maketitle

\section{Introduction}

Infinite cyclic covers are fundamental objects of study in topology (e.g., in knot theory \cite{Ro}, but see also \cite{MilnorInfCyclic}) and algebraic geometry (e.g., for the study of Alexander-type invariants of complex hypersurface complements, see \cite{DL, DN, L1, L3, MaximComment}). 

The Milnor fiber of a hypersurface singularity germ (cf. \cite{Milnor}), 
can be viewed as an example of an infinite cyclic cover,
since it is a retract of the infinite cyclic cover of the complement 
to the germ in a small ball about the singular point. Moreover, 
in this interpretation, the monodromy of the Milnor fiber 
corresponds to the action of the  
generator of the group of  deck transformations of the infinite cyclic cover 
(cf. Section \ref{s2} below; but see also \cite{L2}, where such an identification was used 
to define an abelian version of the Milnor fiber,  
and \cite{Di} for a detailed discussion in the 
homogeneous case). 

Motivated by connections between the Igusa zeta functions, Bernstein-Sato polynomials and the topology of hypersurface singularities, Denef and Loeser defined in \cite{DenefLoeserTopology,DenefLoeserJAG,Barcelona} the {motivic} zeta function and the motivic Milnor fiber of a hypersurface singularity germ; the latter is a virtual variety endowed with an action of the group scheme of roots of unity, from which one can retrieve several invariants of the (topological) Milnor fiber, e.g., the Hodge-Steenbrink spectrum, Euler characteristic, etc. The motivic Milnor fiber has also appeared in the Soibelman-Kontsevich theory of motivic Donaldson-Thomas invariants.

In this paper, we attach to an infinite cyclic cover associated to a 
punctured neighborhood of a simple normal crossing divisor $E$ on a complex quasi-projective manifold $X$, an element in the Grothendieck ring 
$K_0({\rm Var}^{\hat \mu}_{\mathbb{C}})$ of algebraic $\mathbb{C}$-varieties
endowed with a good action of the pro-finite group ${\hat \mu}=\lim \mu_n$ of roots of unity, which we call a {\it motivic infinite cyclic cover};
see Section \ref{s3} for details. (Our terminology is inspired by 
 the standard notion of ``motivic Milnor fiber'', cf. \cite{Barcelona}.) 
Among other consequences, this construction allows us to define a motivic infinite cyclic cover of a hypersurface singularity germ complement, which as we show later on coincides (in the localization of $K_0({\rm Var}^{\hat \mu}_{\mathbb{C}})$ at the class $\mathbb{L}$ of the affine line) with the above-mentioned Denef-Loeser motivic Milnor fiber. Our class of coverings guarantees certain finiteness conditions (see Definition \ref{defft}) which are present in the case of Milnor fibers, but which  are not satisfied in general. 
Note that while these infinite cyclic covers are complex manifolds, they are not algebraic varieties in general. This paper provides an algebro-geometric interpretation of such covering spaces.

 Our construction of motivic infinite cyclic covers is topological in the sense 
that it does not make use of arc spaces as is the case in earlier 
constructions of motivic Milnor fibers. We rely instead on the weak factorization theorem \cite{AKMW,Bon}. 
One of our main results, Theorem \ref{Welldefined}, shows that our notion of motivic infinite cyclic cover is a birational invariant, or equivalently, it is an invariant of the punctured neighborhood of $E$ in $X$.
Moreover, in Section \ref{s4} we show that the Betti realization of the motivic infinite cyclic cover is given by the cohomology with compact support of the infinite cyclic cover of the punctured neighborhood, e.g., their Euler characteristics coincide. 

Finally, in Sections \ref{relationwithmotivic} and \ref{s6}, we explain how the present  construction of a motivic infinite cyclic cover  generalizes the above-mentioned notion of motivic Milnor fibre of a hypersurface singularity germ, as well as the notion of motivic Milnor fiber of a rational function (compare with \cite{R2}). 

\medskip

\textbf{Acknowledgments.}  The authors would like to thank J\"org Sch\"urmann for his interest and comments on a preliminary version of this work. M. Gonz\'alez Villa is partially supported by the grant MTM2010-21740-C02-02 from Ministerio de Ciencia e Innovacion of the Goverment of Spain. A. Libgober is partially supported by a grant from the Simons Foundation. L. Maxim is partially supported by grants from NSF (DMS-1304999), NSA (H98230-14-1-0130), and by a grant of the Ministry of National Education, CNCS-UEFISCDI project number PN-II-ID-PCE-2012-4-0156.

%%%%%%%%%%%%%%%%%%%%%%%%%%%%%%%%%%%%%%

\section{Infinite cyclic cover of finite type}\label{s2}

Let $X$ be a smooth complex quasi-projective variety and $E$ an algebraic (reduced) simple normal crossing divisor on $X$ which shall be called a {\it deletion} 
(or {\it deleted) divisor}. 
Assume that $E = \sum_{i \in J} E_i$ is a decompostion of $E$ into irreducible components $E_i$, where we assume that all divisors $E_i$ are smooth. 
We use the following natural stratification of $X$ given by the intersections of the irreducible components of $E$: for each $I \subseteq J$ consider
\begin{equation}\label{str}E_I = \bigcap_{i \in I}E_i  \quad \hbox{ and } \quad  E^\circ_I = E_I \setminus \bigcup_{j \not \in I} E_j.\end{equation}
Clearly, $X = \bigcup_{I \subseteq J} E^\circ_I$, $X\setminus E= E^\circ_\emptyset$ and $E=\bigcup_{\emptyset \ne I \subseteq J} E^\circ_I$. 

Let $T^*_{X, E}$ be a punctured neighborhood of $E$ on $X$.  Sometimes we omit the subscript $X$ and just 
write $T^*_E$. We recall here the construction of such a punctured neighborhood. For each smooth irreducible component $E_i$ of $E$ ($i \in J$), we choose a tubular neighborhood $T_{E_i} \to E_i$, and define the corresponding neighborhood of $E_I$ (with $\emptyset \neq I \subseteq J$) by: $$T_{E_I}:=\bigcap_{i \in I} T_{E_i}.$$ We set $$T_{X,E}:=\bigcup_{\emptyset \neq I \subseteq J} T_{E_I}.$$ Note that if the chosen tubular neighborhoods $T_{E_i}$ of the components $E_i$ are small enough, then $T_{E_I} \to E_I$ is also a tubular neighborhood for the submanifold $E_I$ (for a suitable projection map), and $T_{X,E}$ is a regular neighborhood of $E$, i.e., $E$ is a deformation retract of $T_{X,E}$. Moreover, the germs of all these neighborhoods (and projection maps) are independent of all choices by the corresponding uniqueness result for  $T_{E_i}$. We define punctured tubular neighborhoods of the strata $E^\circ_I$ by: $$T^*_{E^{\circ}_I}:=(T_{E_I}\vert_{E^{\circ}_I}) \setminus \bigcup_{i \in I} E_i,$$ and the punctured tubular neighborhood of $E$ in $X$ is then given  by:
$$T^*_{X, E}:=\bigcup_{\emptyset \neq I \subseteq J} T^*_{E^{\circ}_I}.$$
By construction, the homotopy types of the (germs of the) punctured neighborhood $T^*_{X,E}$ and projection map $T^*_{E^\circ_I} \rightarrow E^\circ_I$ are well-defined (i.e., independent of all choices). 
Moreover, $T^*_{X, E}$ is a union of locally trivial topological fibrations  $T^*_{E^\circ_I} \rightarrow E^\circ_I$ over the strata $ E^\circ_I$ (with $\emptyset \ne I \subseteq J$) of $E$, the fiber of the latter fibration being homeomorphic to $(\mathbb{C}^\ast)^{ | I |}$, where $| I |$ denotes the number of elements in the set $I$. 

Note that the punctured neighborhood $T^*_{X,E}$ is homotopy equivalent to the boundary of the regular neighborhood $T_{X,E}$, which sometimes is called the {\it link of $E$ in $X$}.
 
\begin{dfn}\label{defft}{\it Infinite cyclic cover of finite type.} 
\\ Let $\Delta: \pi_1(T^*_{X,E}) \twoheadrightarrow \ZZ$ be an epimorphism\footnote{The surjectivity assumption is made here solely for convenience (in which case the corresponding infinite cyclic cover is connected), all results in this paper being valid for arbitrary homomorphisms to $\ZZ$. The only instance when non-surjective homomorphisms are considered is in Section \ref{relationwithmotivic}.},
and denote by  $\widetilde{T}^{*}_{X, E, \Delta}$ the corresponding infinite cyclic cover (with Galois group $\ZZ$) of the 
punctured neighborhood $T^*_{X,E}$ of a simple normal crossing divisor $E \subset X$.
For any $i \in J$, let $\delta_i$ be the boundary 
of a small (oriented) disk  
transversal to the irreducible component $E_i$. 
We call the infinite cyclic cover $\widetilde{T}^{*}_{X, E, \Delta}$ 
{\it of finite type} if $m_i=\Delta(\delta_i)\neq 0$ for all $i \in J$ 
(see Proposition \ref{ft} below for a justification of the terminology).
\end{dfn}

\begin{remark} 
The surjectivity of the restriction of $\Delta$ 
on the kernel of $\pi_1(T^*_E) \rightarrow \pi_1(T_E)=\pi_1(E)$ 
is equivalent to the condition ${\rm gcd}(m_i | i \in J)=1$.
Sometimes we omit $\Delta$ and $X$ in the notation 
and write simply $\widetilde{T}^*_E$. 
The map $\Delta$ will also be referred to as the {\it holonomy}
of this infinite cyclic cover. Note also that $\Delta$ factors through $H_1(T^*_E)$ , so the infinite cyclic covering $\widetilde{T}^{*}_{X, E, \Delta}$depends only on the epimorphism $H_1(T^*_E) \to \ZZ$. Therefore, in the following we can assume that $E$ and $T^*_E$ are connected, and the choice of the basepoint for $\pi_1(T^*_E)$ has no relevance.
\end{remark}

The infinite cyclic cover $\widetilde{T}^{*}_{X, E, \Delta}$ has the structure of complex manifold, but it is not an algebraic variety. In the following section, we will give an algebro-geometric (motivic) realization of $\widetilde{T}^{*}_{X, E, \Delta}$. The type of algebraic structure we consider
 is specified 
further in the following definition.

\begin{dfn}\label{equivnbhd}
 Let $T_1=T_{X_1,E_1}$ and $T_2=T_{X_2,E_2}$ be two regular 
neighborhoods of normal crossing divisors, and 
$\Delta_i: \pi_1(T_i^*) \rightarrow \ZZ$, $i=1,2$, be 
surjections on the fundamental groups of the corresponding punctured
neighborhoods.
We say that $(T_1^*,\Delta_1)$ and $(T_2^*,\Delta_2)$ 
are equivalent if there exist a birational 
map $\Phi: X_1 \rightarrow X_2$, which is regular on $T_1^* \subset X_1$ 
(and respectively, $\Phi^{-1}$ is regular on $T_2^*\subset X_2$), 
and which moreover induces a map $\Phi\vert_{T_1^*}: T_1^* \rightarrow T_2^*$ 
such that $\Phi(T_1^*)$ and $T_2^*$ are deformation retracts of 
each other and  the diagram:
\begin{equation}\label{commutativity}
\begin{matrix}  \pi_1(T_1^*)& &\buildrel (\Phi\vert_{T^*_1})_* \over 
  \longrightarrow & & \pi_1(T_2^*) \cr
                            &\Delta_1\searrow & & \Delta_2\swarrow  &\cr 
                            & &  \ZZ  & & \cr              
\end{matrix}
\end{equation}
is commutative. Here $(\Phi\vert_{T^*_1})_*$ is the homomorphism induced by  $\Phi\vert_{T_1^*}$ on the 
fundamental groups.
\end{dfn}

The following result justifies the terminology used in Definition \ref{defft}. We will use rational coefficients, unless otherwise stated.
\begin{prop}\label{ft} Let $\widetilde{T}^*_E$ be an infinite cyclic 
cover of finite type (as in Definition \ref{defft}). Then for any $i \in \ZZ$, the rational vector spaces $H_c^i(\widetilde{T}^*_E)$ and $H^i(\widetilde{T}^*_E)$ are finite dimensional. Moreover, these cohomology groups are trivial for $\vert i \vert$ large enough.
\end{prop}

\begin{proof} 
We begin by discussing the case of $H_c^i(\widetilde{T}^*_E)$. First note that, under the action of the group $\mathbb{Z}$ of deck transformations, the cohomology groups $H_c^i(\widetilde{T}^*_E)$ become in an usual way $\mathbb{Q}[\mathbb{Z}] \simeq \mathbb{Q}[t,t^{-1}]$-modules.
Then it suffices to show  that $H_c^i({T}^*_E;\mathscr{L})$ is a finite dimensional rational vector space, where $\mathscr{L}$ is the local coefficient system on $T^*_E$ with stalk $\mathbb{Q}[t,t^{-1}]$ corresponding to the representation defined on the meridians $\delta_i$ by $\delta_i \mapsto t^{m_i}$, $i \in J$. 

Recall that  $T^*_E$ is a union of locally trivial fibrations  $T^*_{E^\circ_I} \rightarrow E^\circ_I$ over the strata $ E^\circ_I$ (with $\emptyset \ne I \subseteq J$) of $E$, the fiber of the latter fibration being homeomorphic to $(\mathbb{C}^\ast)^{ | I |}$, where $| I |$ denotes the number of elements in the set $I$. Moreover,  $T^*_E$ has an open cover consisting of the sets $\{T^*_{E^\circ_i}\}_{i \in J}$, with intersections given by $\bigcap_{i \in I} T^*_{E^\circ_i}=T^*_{E^\circ_I}$.
So by the associated Mayer-Vietoris spectral sequence, it suffices to show that each  vector space $H_c^i(T^*_{E^\circ_I};\mathscr{L})$ (with the induced local coefficients) is finite dimensional.

The claim follows from the Leray spectral sequence for the fibration $T^*_{E^\circ_I} \rightarrow E^\circ_I$, i.e.,  $$E_2^{p,q}=H_c^p(E_I^{\circ};\mathcal{H}_c^q((\mathbb{C}^\ast)^{|I|};\mathscr{L})) \Longrightarrow H_c^{p+q}({T}^*_{E_I^{\circ}};\mathscr{L})$$
since the (stalk of the local) coefficients $H_c^q((\mathbb{C}^\ast)^{|I|};\mathscr{L})$ appearing in the $E_2$-term are torsion $\mathbb{Q}[t,t^{-1}]$-modules, hence finite dimensional vector spaces. Indeed, the torsion property follows from the assumption that $m_i \neq 0$, for all $i \in J$.

The case of  $H^i(\widetilde{T}^*_E)$ follows now by Poincar\'e duality. \end{proof}

\begin{remark}\label{1.4} The above proof shows in fact that the cohomology groups $H_c^i(\widetilde{T}^*_E)$ and $H^i(\widetilde{T}^*_E)$ are torsion $\mathbb{Q}[t,t^{-1}]$-modules of finite type. Since $\mathbb{Q}[t,t^{-1}]$ is a principal ideal domain, it follows that $H_c^i(\widetilde{T}^*_E)$ has a well-defined associated order $\Delta_i(t)$, called the {\it $i$-th Alexander polynomial} of $E$, see \cite{MilnorInfCyclic}. Note that $\Delta_i(t)$ can be identified with the characteristic polynomial $\det(t\cdot{\rm Id}-T_i^*)$ of the (monodromy) action induced by the generating deck transformation $T$ on $H_c^i(\widetilde{T}^*_E)$. Then it follows from the arguments used in the proof of Proposition \ref{ft} that all roots of the Alexander polynomials $\Delta_i(t)$ are roots of unity, so in particular, the semi-simple part of $T^*_i$ is a finite order automorphism.
\end{remark}

\begin{remark} Proposition \ref{ft} motivates our search for a ``motive'' (in the sense of Section \ref{s3}), whose Betti realization is that of $\widetilde{T}^*_E$ (cf. Section \ref{s4}).  Note that if instead of the punctured neighborhood $T^\ast_E$ of $E$ in $X$ we consider the complement $X \setminus E$, then the corresponding   infinite cyclic cover is in general not of finite type, in the sense that some of its cohomology groups can be infinite dimensional. In the case of complements to projective hypersurfaces see \cite{DL, MaximComment} for such an example. 
\end{remark}

Let us now consider the local situation of the hypersurface singularity germ which, together with the work of Denef-Loeser about the motivic Milnor fiber, inspired our Definition \ref{defft} and the results of the following sections. 

Let $f(x_1,\cdots,x_n)=0$ define a  hypersurface singularity germ at the origin in $\CC^n$. Let $B_\epsilon$ be a small enough ball about the origin in $\CC^n$  and $D^\ast_\delta$ a small punctured disc in $\CC$, for $0 < \delta << \epsilon$. Set $B_{\epsilon,\delta}:=B_\epsilon \cap f^{-1}(D^\ast_\delta)$ and let $F=\{f=0\} \cap B_\epsilon$. By Milnor's fibration
theorem \cite{Milnor}, one has a locally trivial topological fibration
$\pi: B_{\epsilon,\delta} \rightarrow D^\ast_\delta$, whose fiber  $M_f$
is called the {\it Milnor fiber} of $f$ at the origin. If $\exp:  \RR \rightarrow S^1 \simeq D^\ast_\delta$ is the universal covering map,
then the fiber product $B_{\epsilon,\delta} \times_{D^\ast_\delta} {\RR}$ formed
by using the above maps  $\pi$ and $\exp$ is the infinite cyclic cover
of $B_{\epsilon,\delta}$ corresponding to the homomorphism 
$\pi_1(B_{\epsilon,\delta}) \rightarrow \pi_1(D^\ast_\delta)=
{\ZZ}$
given by the linking number with $F$.  
The covering map is just the
projection of the fiber product on the first factor.
Note that if $f=\prod_i f_i^{m_i}$ is the decomposition of the germ $f$ as a product of distinct irreducible factors, the linking number homomorphism is defined by mapping the meridian generators $\delta_i$ of $\pi_1(B_{\epsilon,\delta})$ to $m_i \in \mathbb{Z}$, with $m_i \neq 0$ for each $i$. Moreover, if $m=\gcd(m_i)_i$, this infinite cyclic cover has exactly $m$ connected components.
On the other hand, the second projection of
$B_{\epsilon,\delta} \times_{D^\ast_\delta} {\RR} \rightarrow {\RR}$ has the same fiber
over $r \in {\RR}$
as the Milnor fibration has over $\exp(r)$. Since ${\RR}$ is
contractible, we obtain the homotopy equivalence between the
infinite cyclic cover of $B_{\epsilon,\delta}$ and the Milnor fiber $M_f$, hence this (local) infinite cyclic cover is of finite type (since $M_f$ is so).
Note that under this identification the monodromy of $\pi$ corresponds
to the deck transformation of the infinite cyclic cover, see also \cite[p. 106-107]{Di}, and \cite{L2}.

%%%%%%%%%%%%%%%%%%%%%%%%%%%%%%%%%%%%

\section{Motivic infinite cyclic covers}\label{s3}
Most of our calculations will be done in the Grothendieck ring $K_0({\rm Var}^{\hat{\mu}}_{\CC})$ of  the category ${\rm Var}^{\hat{\mu}}_{\CC}$ of complex algebraic varieties  endowed with good $\hat{\mu}$-actions. Let us briefly recall the relevant definitions, e.g., see \cite{Barcelona}.

For a positive integer $n$, we denote by $\mu_n$ the group of all $n$-th roots of unity (in $\mathbb{C}$).  The groups $\mu_n$ form a projective system with respect to the maps   $\mu_{d \cdot n} \to \mu_n$ defined by $ \alpha \mapsto \alpha^d$, and we denote by $\hat \mu:=\lim \mu_n$ the projective limit of the $\mu_n$.

Let $X$ be a complex algebraic variety. A {\it good $\mu_n$-action} on $X$ is an algebraic action $\mu_n \times X \to X$, such that each orbit is contained in an affine subvariety of $X$. (This last condition is automatically satisfied if $X$ is quasi-projective.) A {\it good $\hat\mu$-action} on $X$ is a $\hat\mu$-action which factors through a good $\mu_n$-action, for some $n$.

The Grothendieck ring $K_0({\rm Var}^{\hat{\mu}}_{\CC})$ of  the category ${\rm Var}^{\hat{\mu}}_{\CC}$ of complex algebraic varieties  endowed with a good $\hat{\mu}$-action is generated by classes $[Y, \sigma]$ of isomorphic varieties endowed with good $\hat\mu$-actions, modulo  the following relations:
\begin{itemize}
\item[(i)] {\it scissor relation:} \begin{equation}\label{sci} [Y, \sigma] = [Y \setminus Y', \sigma_{|_{Y \setminus Y'}}] + [Y', \sigma_{|_{Y'}}],\end{equation}  if $Y'$ is a closed $\sigma$-invariant subset of $Y$.
\item[(ii)] {\it product relation:} \begin{equation} [Y \times Y', (\sigma, \sigma')] = [Y, \sigma][Y', \sigma'].\end{equation}
\item[(iii)] \begin{equation}\label{r3} [Y \times \mathbb{A}^1_\mathbb{C} , \sigma] = [Y \times \mathbb{A}^1_\mathbb{C}, \sigma'],\end{equation}  if $\sigma$ and $\sigma'$ are two affine liftings of the same $\mathbb{C}^\ast$-action on $Y$.
\end{itemize} 
The third relation above is included for completeness, though it is not needed in this paper.
We denote by $\mathbb{L}$ the class in $K_0({\rm Var}^{\hat{\mu}}_{\CC})$ of $\mathbb{A}^1_\mathbb{C}$, with the trivial $\hat\mu$-action.

\medskip

The following topological lemma provides the crucial ingredients for our definition of motivic infinite cyclic covers.

\begin{lemma}\label{techlem} 
Let $A,B,C$ be connected topological spaces
and let $A \rightarrow B$ be a locally trivial topological
fibration with fiber $C$, so 
we have an exact sequence 
$$\pi_1(C)   \rightarrow \pi_1(A) 
 \rightarrow \pi_1(B)
\rightarrow 0.$$ Let $G$ 
be a group and let $\widetilde A$ be the covering 
space of $A$ with  Galois group $G$ and holonomy map
$\alpha: \pi_1(A) \rightarrow G$.
Then $\widetilde A$ is 
a disjoint union of $[G: {\rm Im} \alpha]$ homeomorphic connected components.
Let $H$ be the image of composition 
$\pi_1(C)  \rightarrow \pi_1(A) 
\buildrel \alpha  \over 
\rightarrow G$ and denote by  $\widetilde C$  the corresponding covering of $C$
with Galois group $H$.  
Then there is a locally trivial topological fibration $\widetilde A \rightarrow \widetilde B$ 
with connected fiber $\widetilde C$,  
where $\widetilde B$ is the cover of $B$ corresponding to the 
map $\pi_1(B)=\pi_1(A)/{\rm Im} \pi_1(C) \rightarrow G/H$.
The number of (homeomorphic) 
connected components of $\widetilde B$ is equal to the index 
$[G: {\rm Im} \alpha]$.

If $A \cong B \times C \to B$ is a trivial fibration, then also $\widetilde A \cong \widetilde B \times \widetilde  C \rightarrow \widetilde B$ is the projection of a trivial fibration (by using the group isomorphism $\pi_1(A) \cong \pi_1(B) \times \pi_1(C)$).
 
\end{lemma}

\begin{proof}

Clearly the map $A \rightarrow B$ induces the map 
of covering spaces $\widetilde A \rightarrow \widetilde B$.
Indeed, one can view $\widetilde A$ as $\widetilde A'\times_{\Im \alpha}G$,
where $\widetilde A'$ 
is the space of paths 
with initial point at the base point of $A$ modulo the equivalence
relation that identifies paths with the same end point such that 
 the corresponding loop belongs to ${\rm Ker}(\alpha)$. 
The action of $\Im \alpha$ on $\widetilde A'$ follows from this description 
of $\widetilde A'$ and the action of $\Im \alpha$ on $G$ is via left multiplication. 
The space $\widetilde{A}$ can be viewed as the set of equivalences classes of 
pairs $(a,g), a \in \widetilde{A'},g \in G$, such that $(a_1,g_1)$ and $(a_2,g_2)$ are 
equivalent iff there exists 
$h \in \Im \alpha $ 
such that $a_1=ha_2, g_1=hg_2$.
The group $G$ acts freely on $\widetilde A$ via action on the second factor and one has the canonical 
identification $\widetilde A/G=A$.

Next, we apply the same construction 
to the homomorphism $$\pi_1(B)=\pi_1(A)/\Im\pi_1(C)\rightarrow G/H$$
to obtain the covering space $\widetilde B$ of $B$ with covering group
$G/H$.
Writing $\widetilde B$ as a fiber product of a path space as above, 
one sees that the map from the space of paths of $A$ to the space of paths of $B$ 
starting at the respective base points induced by the 
map $A \rightarrow B$, is compatible 
with the above mentioned equivalences. 
Thus, we have a map $\widetilde{A} \rightarrow \widetilde{B}$. 

The stabilizer of the fiber $\widetilde C$ of $\widetilde A \rightarrow \widetilde B$
is $H$.  Finally the $G$-orbit of any point in $\widetilde C$ intersects 
$\widetilde C$ in its $H$-orbit. Hence $\widetilde C/H=C$.
\end{proof}

We shall apply the constructions of the Lemma \ref{techlem}
to describe certain covering 
spaces associated with the punctured neighborhoods of 
strata of normal crossings divisors.

 \begin{lemma}\label{definingcovers}
 Let $(X,E)$ be as in the beginning of Section \ref{s2} and let $I \subseteq J$ such that $| I|=r$. The projection of the punctured 
neighborhood $T^\ast_{E^\circ_I}$ onto the stratum $E_I^{\circ}$ induces  an exact sequence 
\begin{equation}\label{normalfundgroupseq} 
\mathbb{Z}^r \rightarrow \pi_1(T^\ast_{E^\circ_I}) \rightarrow \pi_1(E^\circ_I) \rightarrow 0.
\end{equation}
Let $\Delta: \pi_1(T^\ast_E)\rightarrow \ZZ$ be a homomorphism 
onto an infinite cyclic group as in Definition \ref{defft}.
Let $N$ be the index of the image subgroup $\Delta(\pi_1(T^\ast_{E_I^\circ}))$ in $\ZZ$
and, similarly, let $M$ be the index of the image of $\ZZ^r$ in $\ZZ$.
Let  \begin{equation}\label{maptildeEI}
\Delta_I : \pi_1(E^\circ_I) \rightarrow \ZZ/{M\ZZ} \end{equation} be the map induced by $\Delta$ 
according to Lemma \ref{techlem}.
Then the corresponding covering  $\widetilde{E^\circ_I} \rightarrow E^\circ_I$ induced by $\Delta_I $ 
has $N$ connected components, each being the cyclic cover
of $E_I^\circ$ with the covering group $N\ZZ/{M\ZZ}$. Moreover, the infinite cyclic cover $\widetilde{T}^\ast_{E^\circ_I}$ (defined by $\ker(\Delta)$) fibers over $\widetilde{E}^\circ_I$, with fiber $(\mathbb{C}^*)^{r-1}$.

 \end{lemma}

\begin{proof} The exact sequence (\ref{normalfundgroupseq})
is derived from the long exact sequence of homotopy groups  associated to a locally trivial  topological fibration, by using the connectivity of the fibre. 
Indeed, the fiber of 
 $$T^\ast_{E^\circ_I} \to {E^\circ_I} $$
 is diffeomorphic to $(\mathbb{C}^\ast)^r$, hence $\pi_1((\mathbb{C}^\ast)^r)= \mathbb{Z}^r$ and $\pi_0((\mathbb{C}^\ast)^r)=0$. 
 
 Let us discuss the second statement. First note that the image ${\rm Im}( \mathbb{Z}^r \rightarrow \pi_1(T^\ast_{E^\circ_I}))$ is the subgroup generated by the meridians $\delta_i$ with $i \in I$. Moreover, $$\Delta({\rm Im}( \mathbb{Z}^r \rightarrow \pi_1(T^\ast_{E^\circ_I}))) = M \ZZ,$$ with
 $$M=m_I:=\gcd(m_i \ | \ i \in I).$$
Since the homomorphism $\pi_1(T^\ast_{E^\circ_I}) \rightarrow \pi_1(E^\circ_I)$ is surjective, it follows that the homomorphism  $\Delta : \pi_1(T^\ast_{E^\circ_I}) \rightarrow \ZZ$ factors through $\pi_1(E^\circ_I)$. Hence we get a well-defined map $$\Delta_I : \pi_1(E^\circ_I) \twoheadrightarrow N\ZZ/{M\ZZ}$$ given by $\Delta_I(\epsilon) = \Delta(\delta)$ for any $\epsilon \in \pi_1(E^\circ_I)$ and any $\delta \in \pi_1(T^\ast_{E^\circ_I})$ such that $\delta \mapsto \epsilon$.
 
 Finally, by Lemma \ref{techlem}, the long exact sequence of homotopy groups associated to the $(\CC^{\ast})^{r}$-fibration $T^\ast_{E^\circ_I} \to {E^\circ_I}$ induces a locally trivial  topological fibration 
\begin{equation}\label{fibr}\widetilde{T}^\ast_{E^\circ_I} \rightarrow \widetilde{E}^\circ_I,\end{equation}
with connected fiber $\widetilde{(\mathbb{C}^\ast)^{r}} \simeq (\CC^*)^{r-1}$, the infinite cyclic cover of $(\CC^{\ast})^{r}$ defined by the kernel of the epimorphism $\ZZ^{r} \twoheadrightarrow m_I \ZZ$ induced by the holonomy map $\Delta$. 

We conclude the proof by noting that the definition of the map (\ref{maptildeEI}) and the finite (algebraic) covering $\widetilde{E}^\circ_I$ depend on the (homotopy class of the) projection $T^\ast_{E^\circ_I} \to E^\circ_I$, but they are nevertheless intrinsic objects associated to our  context.
 \end{proof}
  
 \begin{dfn}\label{ucov} We denote by $\widetilde{E^\circ_I}$ the unramified cover of $E^\circ_I$ with Galois group $\ZZ/{M\ZZ}$ defined by the map $(\ref{maptildeEI})$, and with $M=m_I:=\gcd(m_i \ | \ i \in I)$. The cover $\widetilde{E^\circ_I}$ is an algebraic variety with a good   $\mu_M$-action  $\sigma_I$ such that $E^\circ_I = \widetilde{E^\circ_I} / \mu_{M}$. It has $N$ 
connected components. The fundamental group $\pi_1(\widetilde{E^\circ_I})$ is isomorphic to ${\rm Ker}(\Delta_I)$.\end{dfn}

\begin{remark}\label{gendef}
The proof of Lemma \ref{definingcovers} applies to the following more general situation. Let $\mathcal{F} \to W$ be a vector bundle on a quasi-projective manifold $W$, and let $\{E_i \subset \mathcal{F} \ \vert \ i \in I \}$ be a collection of $\vert I \vert \geq 1$ independent sub-bundles of $\mathcal{F}$ of corank $1$ (in particular, the collection $\{E_i \ \vert \ i \in I\}$ forms a normal crossing divisor in $\mathcal{F}$). Then one has a locally trivial  topological  fibration $\mathcal{F}^*:=\mathcal{F} \setminus \cup_{i \in I}  E_i\to W$ with fiber $F$ homotopy equivalent to $(\mathbb{C}^*)^{|I|}$. Moreover, a homomorphism $\pi_1(\mathcal{F}^*) \to \mathbb{Z}$ with image $N \ZZ$ and so that the image of $\pi_1(F)$ is a subgroup of finite index $M$ in $N \ZZ$, defines an infinite cyclic cover $\widetilde{\mathcal{F}}^*$ of $\mathcal{F}^*$ having $N$ connected components, each of which is a locally trivial topological fibration with fiber equivalent to $(\mathbb{C}^*)^{|I|-1}$ and base $\widetilde{W}$ being an $M$-fold cyclic cover of $W$.
\end{remark}
    
We now have all the ingredients for defining the main object of the paper:

\begin{dfn}\label{maindef} {\it Motivic infinite cyclic cover}
\newline Let $T_E^*$ be a punctured neighborhood 
of a normal crossing divisor in 
a quasi-projective manifold $X$ as in Section \ref{s2}, 
and let $\Delta: \pi_1(T^*_E) \rightarrow \ZZ$ be a surjection such that 
the corresponding infinite cyclic cover $\widetilde{T}_{X,E,\Delta}^*$ 
has a finite type. 
For each fixed subset $A \subseteq J$, we define the corresponding 
motivic infinite cyclic cover (of finite type) of $T^*_E$ as
\begin{equation}\label{defnmotinfcover}S^A_{X, E, \Delta} :=  \sum_{\substack{\emptyset \not = I \subseteq J\\ A \cap I \ne \emptyset}} (-1)^{|I|-1}[\widetilde{E^\circ_I}, \sigma_I](\mathbb{L}-1)^{|I| -1} \in K_0({\rm Var}^{\hat{\mu}}_{\CC}),
\end{equation}
where $\widetilde{E^\circ_I}$ are the covering spaces corresponding 
to $\Delta_I$ in Definition \ref{ucov}.

When $A=J$, we use the notation $S_{X, E, \Delta}$ or $S_{X, E}$.
\end{dfn}

\begin{remark}\label{int} Recall from Lemma \ref{definingcovers} that the infinite cyclic cover 
$\widetilde{T}^\ast_{E^\circ_I}$ of ${T}^\ast_{E^\circ_I}$  is a $(\mathbb{C}^*)^{|I|-1}$-fibration over $\widetilde{E}^\circ_I$. Therefore, one can regard the product $[\widetilde{E^\circ_I}, \sigma_I](\mathbb{L}-1)^{|I| -1}$ appearing in (\ref{defnmotinfcover}) as a ``motive'' of the infinite cyclic cover 
$\widetilde{T}^\ast_{E^\circ_I}$, while the alternating sum on the right-hand side of (\ref{defnmotinfcover}) can be interpreted as the inclusion-exclusion principle for the cover $T^*_E=\bigcup_{\emptyset \neq I \subseteq J} T^*_{E_I^{\circ}}$.
\end{remark}

The main result of this section is the following. 
\begin{theorem}\label{Welldefined}The above notion of motivic infinite cyclic cover is invariant under the equivalence relation described in Definition \ref{equivnbhd}.
\end{theorem}

Since any birational map $X_1 \rightarrow X_2$ providing an equivalence
between punctured neighborhoods (cf. Definition \ref{equivnbhd}) is, 
by the Weak Factorization Theorem \cite{AKMW} (see also \cite{Bon} for the non-complete case), 
a composition of blow-ups and blow-downs, each inducing an equivalence 
between the corresponding punctured neighborhoods,
it suffices to show that the above expression (\ref{defnmotinfcover}) is invariant under blowing up along a smooth center in $E$. Let us consider $$p : X' := Bl_Z X \rightarrow X$$ the blow-up  of $X$ along the smooth center $Z \subset E$ of codimension $\geq 2$ in $X$. Let us denote by $E_*$  the exceptional divisor of the blow-up $p$, which is isomorphic to the projectivized normal bundle over $Z$, i.e.,  $E_* \cong \mathbb{P}(\nu_Z)$. We may also assume that the center $Z$ of the blow-up  
is contained in $E$ and 
has normal crossings with the components of $E$ (cf. \cite[Theorem 0.3.1,(6)]{AKMW}).  Let us denote the preimage  of the divisor $E_i$ in $X'$ by $E'_i$. Denote by $E'$ the normal crossing divisor in $X'$ formed by the $E'_i$ together with $E_\ast$. Denote by $J' = J \cup \{ \ast \}$ the family of indices of the divisor $E'$. For $I \subseteq J$ we denote by $I' \subseteq J'$  the family $I \cup \{ \ast \}$.  Finally, let $A' = A \cup \{\ast\}$.

By the above reduction to the normal crossing situation, we may assume that there is $I \subseteq J$ such that $Z$ is contained in $E_I$. We 
consider the (surjective) homomorphism given by  
the composition $$\Delta' : \pi_1({T}^*_{X', E'}) 
\rightarrow \pi_1(T^*_{X,E}) 
\buildrel \Delta \over
\rightarrow \mathbb{Z}$$ 
resulting from the identification $T^*_{X'E'} \overset{\cong}{\rightarrow} T^*_{X,E}$ 
induced by the blow-down map. 
We have $\Delta'(\delta'_i)= \Delta(\delta_i)=m_i$ ($i \in I$) and $m_*:=\Delta'(\delta_*) = \sum_{i \in I} m_i$, where $\delta'_i$ and $\delta_*$ are the meridians about the components $E'_i$ and $E_*$ of $E'$.  Indeed, the blow-down map 
takes the $2$-disk transversal to $E_*$ (at a generic point) and 
bounded by $\delta_*$, to 
the disk in $X$ transversal to the components $E_i, i\in I$ containing 
$Z$ and disjoint from the remaining components of $E$, i.e., 
one has the relation $\delta_*=\sum_{i \in I} \delta_i$ in $H_1(T^*_E)$.
Note that $\widetilde{T}^{*}_{X', E',\Delta'}$ is of finite type since $\widetilde{T}^*_{X,E,\Delta}$ is so and $T^*_{X'E'} \cong T^*_{X,E}$. Moreover, if $m_* \neq 0$, then by Lemma \ref{definingcovers} and Definition \ref{ucov} applied to $(X',E', \Delta')$ we can define the corresponding motivic infinite cyclic cover by:
\begin{equation}\label{afterbl}S^{A'}_{X', E', \Delta'} := \sum_{\substack{\emptyset \not = K \subset J' \\ K \cap A' \ne \emptyset}} (-1)^{|K|-1}[\widetilde{E^\circ_K}, \sigma_{K}](\mathbb{L}-1)^{|K| -1}.\end{equation}
If $m_*=0$, then Lemma \ref{definingcovers} cannot be applied directly for defining a finite cover (as in Definition \ref{ucov}) of the dense open stratum $E_*^{\circ}$ of the exceptional divisor $E_*$. However, as already pointed out in Remark \ref{int}, the main ingredient needed at this point for the definition of (\ref{afterbl}) is the ``motive'' of the infinite cyclic cover $\widetilde{T}^*_{E_*^{\circ}}$ of the punctured tubular neighborhood of $E_*^{\circ}$. Such a ``motive'' can be defined by making use of Remark \ref{gendef} as follows. First note that ${T}^*_{E_*^{\circ}}$ is a $\CC^*$-fibration over ${E_*^{\circ}}$, which in turn is a (Zariski) locally trivial fibration over the open dense stratum $Z^{\circ}:=Z \cap E^{\circ}_I$ of $Z$, with fiber $\CC^{s-|I|+1}\times (\CC^*)^{|I|-1}$, where $s$ is the codimension of $Z$ in $E$ (see the proof of Proposition \ref{ind} below). Hence ${T}^*_{E_*^{\circ}}$ is a $\CC^{s-|I|+1}\times (\CC^*)^{|I|}$-fibration over $Z^{\circ}$, and Remark \ref{gendef} together with Lemma \ref{techlem}
can now be used to show that the infinite cyclic cover $\widetilde{T}^*_{E_*^{\circ}}$ is a $\CC^{s-|I|+1}\times (\CC^*)^{|I|-1}$-fibration over the $m_I$-fold cover $\widetilde{Z^{\circ}}$ of $Z^{\circ}$ defined as in Lemma \ref{definingcovers}. So, in this case, we can replace the term $[\widetilde{E_*^{\circ}}]$ of (\ref{afterbl}) (which would correspond to the ``motive'' of $\widetilde{T}^*_{E_*^{\circ}}$) by the product $[\widetilde{Z^{\circ}}]\LL^{s-|I|+1}(\LL-1)^{|I|-1}$. 

Finally, note that since punctured neighborhoods remain  unchanged under blow-ups, it is easy to see by using Lemma \ref{techlem} that, if $m_* \neq 0$, the  product $[\widetilde{Z^{\circ}}]\LL^{s-|I|+1}(\LL-1)^{|I|-1}$ coincides in fact with the motive $[\widetilde{E_*^{\circ}}]$, so as it will become clear from the proof of our main theorem it suffices to assume from now on that $m_* \neq 0$.

\medskip

 Theorem \ref{Welldefined} follows now from the following proposition. 
\begin{prop}\label{ind} With the above notations, we have the following equality of motives:
\begin{equation}\label{S=S}S^A_{X, E,\Delta}  = S^{A'}_{X', E',\Delta'} \in K_0({\rm Var}^{\hat{\mu}}_{\CC}).\end{equation}
\end{prop}

\bigskip

Note that we can always restrict the comparison of  motives in Proposition \ref{ind} to strata in the center of blowup and in the exceptional divisor, respectively. Indeed, the blow-up map induces an isomorphism outside the center $Z$, so the strata in $E\setminus Z$ and $E' \setminus E_*$ are in one-to-one isomorphic correspondence; moreover, these isomorphisms can be lifted (e.g., by Lemma \ref{definingcovers}) to the corresponding unramified covers. It also suffices to prove the above result only in the case $A=J$.

The proof of Proposition \ref{ind} is by induction on the dimension of the center of blow-up.

\subsection{Beginning of induction}
Let us consider the following examples  in relation to the starting case of induction, i.e., when the center $Z$ is a point.

\begin{example}\label{A}
Let  $X$ be a surface and let $E_1$ and $E_2$ be two smooth curves  intersecting transversally at a point $P$. Let us consider the blow-up 
$X'=Bl_Z X$ of $X$ at the center $Z=P$. The exceptional divisor is $E_* \cong \PP^1$ and we have  $E_*^\circ \cong \mathbb{C}^*$. 
Let $\delta_{i} \in H_1(T_{E_i^{\circ}}^*,\ZZ)$ ($i=1,2$) be the class of the fiber of the
projection of punctured neighborhood $T_{E_i^{\circ}}^*$ onto the stratum $E_i^{\circ}$.
If $\Delta(\delta_{1})= m_1$ and $\Delta(\delta_{2})=m_2$, the contribution of $P$ to $S_{X, E}$ is $-[\mu_{{\rm gcd}(m_1,m_2)}](\mathbb{L}-1)$ and the contributions of the exceptional divisor $E_*$ to $S_{X', E'}$ are \footnote{For simplicity, here and in the sequel we denote by $\sigma_{\Delta'}$ the good $\hat{\mu}$-action on the corresponding finite cover (defined by using the holonomy $\Delta'$) of a stratum in the exceptional divisor, cf. Definition \ref{ucov}.}
$$[\widetilde{E_*^\circ}, \sigma_{\Delta'}]-\left([\widetilde{E'_1 \cap E_*}, \sigma_{\Delta'}] + [\widetilde{E'_2\cap E_*}, \sigma_{\Delta'}]\right)(\mathbb{L}-1).$$ 
Because ${\rm gcd}(m_1, m_2) = {\rm gcd}(m_1, m_1 + m_2) = {\rm gcd}(m_2, m_1 + m_2)$, we get: $[\widetilde{E'_1 \cap E_*}, \sigma_{\Delta'}] = [\widetilde{E'_2\cap E_*}, \sigma_{\Delta'}] = [\mu_{{\rm gcd}(m_1, m_2)}]$. Finally, Lemma \ref{definingcovers} asserts that $[\widetilde{E_*^\circ}, \sigma_{\Delta'}]=[\mu_{{\rm gcd}(m_1, m_2)}](\mathbb{L}-1)$. 
To see this directly, let us describe explicitly the
covering space of $E_*^\circ$ according to Lemma \ref{definingcovers}.
In the notations of the above-mentioned lemma,  we have that $M=m_1+m_2$ and 
$N=\gcd(m_1,m_2)$.
Indeed, denoting by $\delta_*$ the homology class of the meridian 
about $E_*^\circ$, the 
homomorphism defining the infinite cyclic 
cover of the punctured neighborhood of  $E_*^\circ$ is given by $\delta_i \mapsto m_i$ ($i=1,2$) and 
$\delta_* \mapsto m_1+m_2$. So Lemma \ref{definingcovers} yields that the Galois group 
of the cover $\widetilde{E_*^\circ} \rightarrow E_*^\circ$ is $\ZZ/(m_1+m_2)$ and, moreover, 
$\widetilde{E_*^\circ}$ has $\gcd(m_1,m_2)$ connected components,  each 
being a connected 
cyclic cover of $\CC^*$ (of degree ${m_1+m_2}\over {\gcd(m_1,m_2)}$).
Such a connected cover is biregular to $\CC^*$, i.e., its 
motive is $\LL-1$, hence the motive of $\widetilde{E_*^\circ}$
is $[\mu_{\gcd(m_1,m_2)}](\LL-1)$.
It follows that both  contributions to the motivic infinite cyclic cover coincide.

Note that in the case when $P$ belongs to only one irreducible component, say $E_1$,  we have  $[\widetilde{E_*^\circ}, \sigma_\Delta] = [\mu_{m_1}] \mathbb{L}$. 
In this case, the contribution to $S_{X', E'}$ is $ [\widetilde{E_*^\circ}, \sigma_{\Delta'}]- [\widetilde{E'_1 \cap E_*}, \sigma_{\Delta'}](\mathbb{L} -1)  =[\mu_{m_1}] \mathbb{L}- [\mu_{m_1}](\mathbb{L} -1)  = [\mu_{m_1}]$. This coincides with the contribution of $P$ to $S_{X, E}$ which is $[\widetilde{E}^\circ_1 |_P] = [\mu_{m_1}]$.\hfill$\square$
\end{example}

\begin{example}\label{B} Let $X$ be a threefold and $E_1$, $E_2$, $E_3$ be three divisors 
intersecting transversally at a point $P$. Consider the blow-up $X'$ of 
$X$ at the center $Z=P$. The divisor $E=\sum_i E_i$ of $X$ transforms into the divisor 
$E'$ in $X'$ consisting of  the proper transforms $E'_i$ of the irreducible 
components $E_i$ of  $E$ ($i=1,2,3$),  
together with the exceptional component $E_*\cong \PP^2$. 
As already mentioned, it suffices to restrict the comparison of  motives $S_{X,E}$ and $S_{X',E'}$ only to contributions coming from the strata in the center of blowup and the exceptional divisor, respectively.

The exceptional divisor $E_*$ acquires seven strata induced from the stratification of $E'$. These strata are: 
\begin{itemize}
\item $L_{\{i,j\}}=E_* \cap E'_i \cap E'_j$, for $i,j \in \{1,2,3\}$ with $i \neq j$, 
\item  $L_{\{i\}}=(E_* \cap E'_i) \setminus (L_{\{i,j\}} \cup L_{\{i,k\}})$, with $\{i,j,k\}=\{1,2,3\}$,  
\item $E_*^{\circ}=E_*\setminus \bigcup_{i=1}^3 E'_i$.
\end{itemize}
Note that  the strata $E_*^{\circ}$ and $L_{\{i\}}$ are complex tori of dimension 
$2$ and $1$, respectively, while the strata $L_{\{i,j\}}$ are points.

Let $T^*_{S}$ denote the punctured neighborhood (in $X'$)
of a stratum $S$ in $E_*$, and let 
$N_{S}$ denote the fiber of the projection 
$T^*_{S} \rightarrow S$.
The fibers  $N_{E_*^{\circ}}$, $N_{L_{\{i\}}}$, $N_{L_{\{i,j\}}}$
corresponding to the punctured neighborhoods of the strata in $E_*$
are homotopy equivalent to real tori of 
dimensions $1$, $2$ and $3$, respectively. 
The first homology group $H_1(T^*_{S},\ZZ)$ of the punctured neighborhood of a stratum $S$  
is generated by the image of $H_1(N_{S},\ZZ)$ under the homomorphism $\Delta'$, together with the 
classes of  boundaries of normal disks (i.e., meridians)
to components $E'_i$ which intersect the closure of $S$.
This observation allows us to calculate the image  subgroup 
$\Delta'(\pi_1(T^*_{S}))$, which  for 
each of the seven strata of $E_*$ results in 
\begin{equation}\label{imageDelta}
\Delta'(\pi_1(T^*_{S}))=\gcd(m_1,m_2,m_3)\ZZ.
\end{equation}
Indeed, the images of homomorphisms 
 $$\Delta'_{N_{S}}: H_1(N_{S},\ZZ) \rightarrow 
H_1(T^*_{S},\ZZ) {\to} \ZZ$$ 
for each of the strata in $E_*$  
are given as follows: 
\begin{itemize}
\item $\Im \Delta'_{N_{E_*^{\circ}}}=(m_1+m_2+m_3) \ZZ$, by the definition of the homomorphism $\Delta'$ on the meridian $\delta_*$ about $E_*$,
\item $\Im \Delta'_{N_{L_{\{i\}}}}=
\gcd(m_i, m_1+m_2+m_3)\ZZ$. 
Indeed, $H_1(N_{L_i},\ZZ)$ is generated by the meridian 
about the exceptional component $E_*$ and the meridian about 
$E_i'$ (which also can be viewed as a meridian of $E_i$).
\item Similarly, 
$\Im \Delta'_{N_{L_{\{i,j\}}}}=\gcd(m_i,m_j,m_1+m_2+m_3)\ZZ,$ for $i,j \in \{1,2,3\}$ with $i \neq j$.
\end{itemize}
So, to verify (\ref{imageDelta}) for a stratum $L_{\{i\}}$, we use the fact that the  homomorphism $\Delta'$ factors through the abelianization map, and the fact that 
the image $\Delta'(H_1(T^*_{L_{\{i\}}},\ZZ))$  
is generated by $\gcd(m_i,\sum_{j=1}^3 m_j)$ 
and by the integers $m_j$, $j\in\{1,2,3\} \setminus \{i\}$ (which are the values of the holonomy 
on the boundaries of normal disks to components $E'_j$, $j \ne i$,
which intersect the closure of $L_{\{i\}}$). 
Similarly, the image $\Delta'(H_1(T^*_{E_*^{\circ}},\ZZ))$ is generated by $\sum_{j=1}^3 m_j$ and by the integers $m_i$, $i=1,2,3$, corresponding to the values of the holonomy on the meridians to the components $E'_i$, $i=1,2,3$, all of which intersect the closure $E_*$ of $E_*^{\circ}$. Finally, for a stratum $L_{\{i,j\}}$, the image $\Delta'(H_1(T^*_{L_{\{i,j\}}},\ZZ))$ is generated by $\gcd(m_i, m_j, m_1+m_2+m_3)=\gcd(m_1,m_2,m_3)$.

It follows from Lemma \ref{definingcovers} 
that
for each of seven strata of $E_*$, the corresponding unbranched covers of Definition \ref{ucov} 
 have $\gcd(m_1,m_2,m_3)$ components, each of which is biregular to the stratum itself (since all these strata are tori).
Hence the contribution of $E_*$ to $S_{X',E'}$ is:
$$[\mu_{\gcd(m_1,m_2,m_3)}]\left((\mathbb{L}-1)^2-3(\mathbb{L}-1)(\mathbb{L}-1)
+3(\mathbb{L}-1)^2\right)
$$
which equals the contribution of $P$ to $S_{X,E}$, given by $[\mu_{\gcd(m_1,m_2,m_3)}]
(\mathbb{L}-1)^2$. \hfill$\square$
\end{example}

\begin{example}\label{C} Let $X$ be a threefold, and $E=E_1+E_2$ be a simple normal crossing divisor on $X$, with holonomy values $m_1$ and resp. $m_2$ on the meridians about its irreducible components. Let $m=\gcd(m_1,m_2)$. Choose a point $Z$ contained in the (one-dimensional) intersection $E_J:=E_1 \cap E_2$, for 
 $J=\{1,2\}$, and consider the blow-up $X'=Bl_Z X$ of $X$ along the center $Z$.  We denote the exceptional divisor $\PP(\nu_Z)$ by $E_*$.  The divisor $E$ is transformed under the blow-up into the divisor $E'$ in $X'$ consisting of the proper transforms $E'_i$ ($i \in J$) of the irreducible components $E_i$ of $E$, together with the exceptional divisor $E_*\cong \PP^2$.
 
 Let us explicitly describe the contribution of the center $Z$ and that of the exceptional divisor $E_*$ to the motives $S_{X,E}$ and $S_{X',E'}$, respectively. 
 Clearly, the class $[\widetilde{E_J}|_Z,\sigma_{\Delta}]$ equals $[\mu_m]$. So the contribution of $Z$ to $S_{X,E}$ consists of $-[\mu_m](\LL-1)$. On the other hand, the exceptional divisor $E_*$ acquires four strata induced from the stratification of $E'$, namely,
 \begin{itemize}
 \item $L_J=E'_1 \cap E'_2 \cap E_*$, which is just a point.
 \item $L_{\{i\}}=E_* \cap E'_i \setminus L_J \cong \CC$, for $i \in I$.
 \item $E^{\circ}_*=E_* \setminus (E_1' \cup E_2') \cong \CC \times \CC^*$.
 \end{itemize}
So the contribution of $E_*$ to  $S_{X',E'}$ is given by:
\begin{equation}\label{pos} [\widetilde{E^{\circ}_*},\sigma_{\Delta'}]-\left( [\widetilde{L_{\{1\}}},\sigma_{\Delta'}] + [\widetilde{L_{\{2\}}},\sigma_{\Delta'}]  \right) (\LL-1) + [\widetilde{L_J},\sigma_{\Delta'}](\LL-1)^2.\end{equation}
Note that, since any of the four strata in $E_*$ is either simply-connected or a product of a simply-connected space with a torus, any finite connected (unbranched) cover of such a stratum is biregular to the stratum itself. So in order to understand the motives of covering spaces appearing in (\ref{pos}), it suffices to compute the number of connected components of each cover. This can be done easily by using Lemma \ref{definingcovers} as follows. First, recall that for a given stratum $S$ of $E_*$, the number of connected components of the corresponding unbranched cover $\widetilde{S}$ (as in Definition \ref{ucov}) equals the index (in $\ZZ$) of the image (under the homomorphism $\Delta'$) of the fundamental group $\pi_1(T^*_S)$ of a punctured neighborhood of $S$ in $X'$. Moreover, since $\Delta'$ factorizes through the abelianization map, it suffices to compute the index $[\ZZ: \Im H_1(T^*_S)]$. Finally, $H_1(T^*_S)$ is generated by $H_1(N_S)$ together with the (classes of) meridians to components $E'_i$ intersecting the closure of $S$, where $N_S$ denotes as before the (normal) fiber of the projection $T^*_S \to S$. In our situation, for each of the above strata in $E_*$, it can be easily seen that the image of $H_1(T^*_S)$ is generated by $m_1+m_2$, $m_1$ and $m_2$, i.e., each of the covering space appearing in (\ref{pos}) has exactly $\gcd(m_1+m_2,m_1,m_2)=m$ connected components. It follows that  (\ref{pos}) can be computed as:
$$[\mu_m]\LL(\LL-1)-2[\mu_m]\LL(\LL-1)+[\mu_m](\LL-1)^2=-[\mu_m](\LL-1),$$
which equals the contribution of $Z$ to $S_{X,E}$.\hfill$\square$
\end{example}

Let us now prove the beginning  case of induction for Proposition \ref{ind}. 
\begin{prop} \label{pt}
The assertion of Proposition \ref{ind} holds in the case when the center of blow-up $Z$ is zero-dimensional.
\end{prop}

It suffices to prove  Proposition \ref{pt} in the case when the center of blow-up is a single point. Indeed, the blow-up at a finite number of points can be regarded as a finite number of single point blow-ups. 

We can thus assume that $Z$ is a point. Let  $r+1=\codim_X Z$, which, by our assumption, equals $\dim X$. Then the exceptional divisor is $E_* \cong \mathbb{P}^r$.   The divisor $E=\sum_i E_i$ of $X$ transforms under the blow-up into the divisor 
$E'$ in $X'$ consisting of  the proper transforms $E'_i$ of the irreducible 
components $E_i$ of  $E$,  
together with the exceptional component $E_*$. 
It suffices to restrict the comparison of  motives $S_{X,E}$ and $S_{X',E'}$ only to contributions coming from the strata in the center of blow-up and the exceptional divisor, respectively.

As in the above examples, we need to describe the stratification of $E_\ast\cong \PP^r$ induced from that of $E'$ (see (\ref{str}) for the latter). Assume that $Z \subseteq \bigcap_{i=1}^k E_i$. We have the following result:

\begin{lemma}\label{strat} For each $k$ with $1 \leq k \leq r+1$ we have the following identity in $K_0({\rm Var}_\CC)$:
\begin{equation}\label{excdivblowuppt}[\mathbb{P}^r] = \sum_{l=0}^{k-1} \binom{k}{l} \mathbb{L}^{r-k+1} (\mathbb{L}-1)^{k-l-1} + [\mathbb{P}^{r-k}].\end{equation}
 The right-hand side describes the stratification of the exceptional divisor $E_* \cong \mathbb{P}^r$ induced by the divisor $\sum_{i=1}^k E'_i$ consisting of the proper transforms of components of $E$ containing the center of blow-up. More precisely, by setting  $K:=\{1,\cdots, k\}$, the strata of $E_*$ are:
 \begin{itemize}
\item  $L_{K}:=(\bigcap_{i=1}^k E'_i) \cap E_*$, which is isomorphic to $\mathbb{P}^{r-k}$. 
\item $\binom{k}{l}$ strata of dimension $r-l$ and of the form $$L_I:=(\bigcap_{i \in I} E'_i) \cap E_*\setminus \bigcup_{i \in K \setminus I} E'_i,$$ with $I \subset K$ and $1 \leq |I|=l \leq k-1$, each of which is isomorphic to $\CC^{r-k+1} \times (\CC^*)^{k-l-1}$.
The  class of each such stratum in $K_0({\rm Var}_\CC)$ is equal to $\mathbb{L}^{r-k+1} (\mathbb{L}-1)^{k-l-1}$.
\item $E^{\circ}_{\ast}:=E_* \setminus \bigcup_{i =1}^k E'_i$, of dimension $r$, which is isomorphic to $\CC^{r-k+1}\times(\CC^*)^{k-1}$, and whose class in $K_0({\rm Var}_\CC)$ is $\LL^{r-k+1}(\LL-1)^{k-1}$.
\end{itemize}
\end{lemma}

\begin{proof}
Note that the stratum $L_K$ is just the projectivization of the normal bundle of $Z$ in $\bigcap_{i=1}^k E_i$, i.e., the exceptional divisor of the blow-up of $Z$ inside $\bigcap_{i=1}^k E_i$. Also, the stratum $E^{\circ}_*$ can be regarded as $L_{\emptyset}$ (i.e., for $l=0$), so all strata can be treated uniformly (see below).

We prove the identity (\ref{excdivblowuppt}) by induction on $k$. For $k=1$ the equality (\ref{excdivblowuppt}) becomes 
$[\mathbb{P}^r]= \mathbb{L}^r + [\mathbb{P}^{r-1}]$, which corresponds to the stratification of the projective space consisting of the affine part and the (projective) hyperplane at infinity. Clearly $[L_{K}]=[\mathbb{P}^{r-1}]$ and $[E^\circ_{\ast}]=\LL^r$. There are no strata of type $L_I$ with $I \subset K$. 

For $k=2$ we are considering a new irreducible component  $E_2$ of $E$ containing $Z$. The class $[E^\circ_{\ast}]$ transforms from $\mathbb{L}^r$ (for $k=1$) to $\mathbb{L}^{r-1}(\mathbb{L}-1)$. Moreover, we have two strata $L_{\{1\}}$ and $L_{\{2\}}$
whose class is $\mathbb{L}^{r-1}$. In this case the equality (\ref{excdivblowuppt}) becomes $[\mathbb{P}^r]= \mathbb{L}^{r-1}(\mathbb{L}-1) + 2\mathbb{L}^{r-1} + [\mathbb{P}^{r-2}]$.

For the general case, when moving from $k-1$ to $k$, the stratum of minimal dimension $r-k+1$ and with class $[\mathbb{P}^{r-k+1}]$ is subdivided into an affine piece $\mathbb{L}^{r-k+1}$ and (the class of) a hyperplane at infinity $[\mathbb{P}^{r-k}]$. Furthermore, each of the $\binom{k-1}{l}$ strata of dimension $r-l$ and with class $\mathbb{L}^{r-k+2} (\mathbb{L}-1)^{k-l-2}$ is subdivided into a piece of dimension $r-l$ and type $\mathbb{L}^{r-k+1} (\mathbb{L}-1)^{k-l-1}$ and a piece of dimension $r-l-1$ and  type $\mathbb{L}^{r-k+1} (\mathbb{L}-1)^{k-l-2}$. Therefore, for the index $k$, the number of strata of dimension $r-l$ and  type $\mathbb{L}^{r-k+1} (\mathbb{L}-1)^{k-l-1}$ is the sum of the $\binom{k-1}{l-1}$ strata coming from  strata of dimension $r-l+1$ for the index $k-1$, and the $\binom{k-1}{l}$ strata of dimension $r-l$ coming from  strata of dimension $r-l$ for the index $k-1$. Therefore, there are 
 $\binom{k}{l}=\binom{k-1}{l-1} + \binom{k-1}{l}$ such strata of dimension $r-l$ for the index $k$. This proves the lemma. 
\end{proof}

\begin{proof} ({\it of Proposition \ref{pt}})\newline
As already pointed out, it suffices to check  the invariance (\ref{S=S}) of motivic infinite cyclic cover under blow-up  in the case when $Z$ is a single point. Assume $Z \subseteq \bigcap_{i=1}^k E_i$.  
Let $K=\{1, 2, \dots, k\}$ and set $m = {\gcd}(m_1,\cdots,m_k)$, where the $m_i$ are the values of the holonomy on the meridians $\delta_i$ about the components $E_i$. Clearly, the class $[\widetilde{E}_K^\circ{}_{|_Z}, \sigma_\Delta]$ equals $[\mu_m]$. Therefore, the corresponding contribution of $Z$ to the  left hand-side of (\ref{S=S}) is $(-1)^{k-1} [\mu_m](\mathbb{L}-1)^{k-1}$. 
On the other hand, for any stratum $S$ of the exceptional divisor $E_*$ (as described in Lemma \ref{strat}), the motive of the corresponding unbranched cover of Definition \ref{ucov} can be computed by: 
\begin{equation}\label{tech}[\widetilde{S},\sigma_{\Delta'}]=[\mu_m][S,\sigma_{\Delta}].\end{equation}
In order to see this, we note that, since any such stratum $S$ in $E_*$ is (by Lemma \ref{strat}) either simply-connected or a product of a simply-connected space with a torus, any finite connected (unbranched) cover of $S$ is biregular to $S$. So in order to prove (\ref{tech}), it suffices to show that the unbranched cover $\widetilde{S}$ of Definition \ref{ucov} has exactly $m$ connected components. 
This can be done by using Lemma \ref{definingcovers} as follows. First, recall that for a given stratum $S$ of $E_*$, the number of connected components of the corresponding unbranched cover $\widetilde{S}$ equals the index in $\ZZ$ of the image (under the  homomorphism $\Delta'$) of the fundamental group $\pi_1(T^*_S)$ of a punctured neighborhood of $S$ in $X'$. Moreover, since $\Delta'$ factorizes through the abelianization map, it suffices to compute the index $[\ZZ: \Im H_1(T^*_S)]$. Finally, $H_1(T^*_S)$ is generated by $H_1(N_S)$ together with the (classes of) meridians to components $E'_i$ intersecting the closure of $S$, where $N_S$ denotes as before the (normal) fiber of the projection $T^*_S \to S$. In our situation, it is easy to see that, for any stratum $S$ in $E_*$, we have:
\begin{equation}\label{index}
[\ZZ: \Im H_1(T^*_S)]=\gcd(\sum_{i \in K} m_i, m_1,\cdots,m_k)=m.
\end{equation} 
Indeed, for any stratum $L_I$ of $E_*$ (where we also include the extremal cases when $I=\emptyset$ or $I=K$), the image of $H_1(N_{L_I},\ZZ)$ is generated by the integers $\sum_{i \in K} m_i$ and $\{m_i, i \in I\}$, while the remaining integers $\{m_i, i \in K \setminus I\}$ correspond to the values of holonomy on the meridians about the components $E'_i$ ($i \in K \setminus I$), all of which intersect the closure of $L_I$.

Taking into account the description of the stratification in Lemma \ref{strat} we have now to check that $(-1)^{k-1} [\mu_m](\mathbb{L}-1)^{k-1}$ equals  $$\left( \sum_{l=0}^{k-1} (-1)^{l} \binom{k}{l} [\mu_m] \mathbb{L}^{r-k+1}(\mathbb{L}-1)^{k-1}  \right) + (-1)^{k}[\mu_m][\mathbb{P}^{r-k}](\mathbb{L}-1)^k.$$
Factoring out  $[\mu_m](\mathbb{L}-1)^{k-1}$ and using that $[\mathbb{P}^{r-k}](\mathbb{L}-1)=\mathbb{L}^{r-k+1}-1$, it remains to show the equality:
\begin{equation}\label{dimZ=0}(-1)^{k-1} = \sum_{l=0}^{k} (-1)^{l} \binom{k}{l} \mathbb{L}^{r-k+1}   + (-1)^{k-1}.\end{equation}
And  (\ref{dimZ=0}) holds because $\sum_{l=0}^k (-1)^{l} \binom{k}{l}=(1-1)^k=0$. This finishes the proof of Proposition \ref{pt}.
\end{proof}

\subsection{Invariance of motivic infinite cyclic covers under blowups: general case}

We begin with a few examples. 

\begin{example}\label{D}
Let $X$ be a threefold and $E=\sum_{i\in J} E_i$, with $J=\{1,2,3\}$, be a simple normal crossing divisor. Let $Z$ be the intersection of $E_1$ and $E_2$. Set $I=\{1,2\}$, so in the notations from the introduction, we have that $Z=E_I$. The component  $E_3$ is transversal to $Z$.
Let us consider the blow-up $X'=Bl_Z X$ of $X$ along the center $Z$. As before we denote the exceptional divisor $\PP(\nu_Z)$ by $E_*$. 

The strata in $Z$ are $E_I^{\circ}=E_1\cap E_2 \setminus E_3$ and the point $E_J=\cap_{i\in J} E_i$, so the contribution of the center $Z$ to the motivic infinite cyclic cover $S_{X,E}$ is: $$-[\widetilde{E^{\circ}_I},\sigma_{\Delta}] (\LL -1)+[\widetilde{E_J},\sigma_{\Delta}](\LL -1)^2$$
The exceptional divisor $E_*$ acquires a stratification with strata of the form:
$$L_I:=(\bigcap_{i \in I} E'_i) \cap E_*\setminus \bigcup_{i \in J \setminus I} E'_i,$$ with $I \subseteq J$, where the dense open stratum $E_*^{\circ}$ in $E_*$ is identified with $L_{\emptyset}$. More precisely, the strata of $E_*$ are in this case the following:
\begin{itemize}
\item $L_{\{1,3\}}=E_* \cap E'_1 \cap E'_3$, $L_{\{2,3\}}=E_* \cap E'_2 \cap E'_3$.
\item  $L_{\{1\}}=(E_* \cap E'_1) \setminus E'_3$, $L_{\{2\}}=(E_* \cap E'_2) \setminus E'_3$, $L_{\{3\}}=(E_* \cap E'_3) \setminus (E'_1 \cup E'_2)$. 
\item $E_*^{\circ}=E_*\setminus \bigcup_{i=1}^3 E'_i$.
\end{itemize}
So the contribution of the exceptional divisor $E_*$ to the motive $S_{X',E'}$ is:
\begin{align*} [\widetilde{E_*^{\circ}},\sigma_{\Delta'}]-\left( [\widetilde{L_{\{1\}}},\sigma_{\Delta'}] +[\widetilde{L_{\{2\}}},\sigma_{\Delta'}] + [\widetilde{L_{\{3\}}},\sigma_{\Delta'}]\right) (\LL-1) \\ + \left( [\widetilde{L_{\{1,3\}}},\sigma_{\Delta'}]+[\widetilde{L_{\{2,3\}}},\sigma_{\Delta'}] \right) (\LL-1)^2.\end{align*}
Note that by Example \ref{A}, applied to the blow-up of the point $E_J$ of intersection of transversal curves $E_1 \cap E_3$ and $E_2 \cap E_3$ inside the surface $E_3$, we have that:
$$-[\widetilde{E_J},\sigma_{\Delta}](\LL -1)=[\widetilde{L_{\{3\}}},\sigma_{\Delta'}]-\left( [\widetilde{L_{\{1,3\}}},\sigma_{\Delta'}]+[\widetilde{L_{\{2,3\}}},\sigma_{\Delta'}] \right) (\LL -1).$$
So in order to show that the contributions of $Z$ and $E_*$ to the motives $S_{X,E}$ and respectively $S_{X',E'}$ coincide, it suffices to prove the equality of motives:
\begin{equation}\label{pr}-[\widetilde{E^{\circ}_I},\sigma_{\Delta}] (\LL -1)=[\widetilde{E_*^{\circ}},\sigma_{\Delta'}]-\left( [\widetilde{L_{\{1\}}},\sigma_{\Delta'}] +[\widetilde{L_{\{2\}}},\sigma_{\Delta'}] \right) (\LL-1).\end{equation}
Next, note that by the definition of blow-up, we have isomorphisms $$L_{\{1\}} \cong E^{\circ}_I \cong L_{\{2\}}$$ which, moreover,  extend (by Lemma \ref{definingcovers}) to the corresponding unbranched covers of Definition \ref{ucov}. Also, by Lemma \ref{techlem}, the (Zariski) locally trivial fibration $E_*^{\circ} \to E^{\circ}_I$ (with fiber $\PP^1 \setminus \{2 \ {\rm points}\}=\CC^*$) can be lifted to a $\CC^*$-fibration $\widetilde{E_*^{\circ}} \to \widetilde{E^{\circ}_I}$. Finally, the Zariski triviality implies that $[\widetilde{E_*^{\circ}},\sigma_{\Delta'}] = [ \widetilde{E^{\circ}_I},\sigma_{\Delta}] (\LL-1)$, which proves the claim. \hfill$\square$
\end{example}

\begin{example}\label{E} Let $X$ be a fourfold and $E= E_1 + E_2 + E_3 + E_4$ be a simple normal crossing divisor on $X$. Let the center $Z$ be the intersection $E_1\cap E_2$ (i.e., $Z= E_{\{1,2\}}$), hence the components $E_3$ and $E_4$ of $E$ intersect $Z$ transversally. 

The center $Z= E_{\{1,2\}}$ is stratified by $E^\circ_{\{1,2\}}$  (open dense stratum), $E^\circ_{\{1,2,3\}}$, $E^\circ_{\{1,2,4\}}$ and $E_{\{1,2,3,4\}}=E^\circ_{\{1,2,3,4\}}$. In the notations of Example \ref{D}, the exceptional divisor $E_\ast$ is stratified by the open dense stratum $E^\circ_{\ast}$, the codimension one strata $L_{\{1\}}$, $L_{\{2\}}$, $L_{\{3\}}$ and $L_{\{4\}}$,  the codimension two strata $L_{\{1,3\}}$, $L_{\{1,4\}}$, $L_{\{2,3\}}$, $L_{\{2,4\}}$ and $L_{\{3,4\}}$, and by the points $L_{\{1,3,4\}}$ and $L_{\{2,3,4\}}$. 

Therefore the invariance under blowup of the motivic infinite cyclic cover is equivalent to the equality of the motives (where, for lack of space, we omit the reference to actions from our notation)
$$-[\widetilde{E}^\circ_{\{1,2\}}](\mathbb{L}-1) + ([\widetilde{E}^\circ_{\{1,2,3\}}]+[\widetilde{E}^\circ_{\{1,2,4\}}])(\mathbb{L}-1)^2-[\widetilde{E}^\circ_{\{1,2,3,4\}}](\mathbb{L}-1)^3$$
and 
\begin{align*} [\widetilde{E}^\circ_\ast] - ([\widetilde{L}_{\{1\}}]+[\widetilde{L}_{\{2\}}])(\mathbb{L}-1)\\ +([\widetilde{L}_{\{1,3\}}]+ [\widetilde{L}_{\{1,4\}}]+[\widetilde{L}_{\{2,3\}}]+[\widetilde{L}_{\{2,4\}}]+[\widetilde{L}_{\{3,4\}}])(\mathbb{L}-1)^2- \nonumber \\([\widetilde{L}_{\{1,3,4\}}]+[\widetilde{L}_{\{2,3,4\}}])(\mathbb{L}-1)^3\end{align*}
respectively.
Note that by transversality, the dimension of the intersection of $Z$ with the components $E_3$ and $E_4$ (and $E_{\{3,4\}}$) is less than the dimension of $Z$.  
So, by induction on the dimension of the center, we can 
assume for the blowup of $E_{\{3,4\}}$ along $E_{\{1,2,3,4\}}=Z \cap E_{\{3,4\}}$, with corresponding deleted divisor $E_{\{1,3,4\}}+ E_{\{2,3,4\}}$, that 
\begin{equation}\label{e34}
-[\widetilde{E}^\circ_{\{1,2,3,4\}}](\mathbb{L}-1)= [\widetilde{L}_{\{3,4\}}] - ([\widetilde{L}_{\{1,3,4\}}]+[\widetilde{L}_{\{2,3,4\}}])(\mathbb{L}-1).\end{equation}
Similarly, for the blowup of $E_{3}$ along the center $E_{\{1,2,3\}}=Z \cap E_3$, with deleted divisor $E_{\{1,3\}}+E_{\{2,3\}}+E_{\{3,4\}}$, we have that 
\begin{equation}\label{e3} 
\begin{aligned}
-[\widetilde{E}^\circ_{\{1,2,3\}}](\mathbb{L}-1)+  [\widetilde{E}^\circ_{\{1,2,3,4\}}](\mathbb{L}-1)^2 \\ 
= [\widetilde{L}_{\{3\}}]-([\widetilde{L}_{\{1,3\}}]+[\widetilde{L}_{\{2,3\}}]+[\widetilde{L}_{\{3,4\}}])(\mathbb{L}-1) \\ + ([\widetilde{L}_{\{1,3,4\}}]+[\widetilde{L}_{\{2,3,4\}}])(\mathbb{L}-1)^2.
\end{aligned}\end{equation}
Finally, for the blowup of $E_{4}$ along the center $E_{\{1,2,4\}}=Z \cap E_4$, with deleted divisor $E_{\{1,4\}}+E_{\{2,4\}}+E_{\{3,4\}}$, we have that 
\begin{equation}\label{e4}
\begin{aligned}
-[\widetilde{E}^\circ_{\{1,2,4\}}](\mathbb{L}-1)+[\widetilde{E}^\circ_{\{1,2,3,4\}}](\mathbb{L}-1)^2
 \\  
 = [\widetilde{L}_{\{4\}}]-([\widetilde{L}_{\{1,4\}}]+[\widetilde{L}_{\{2,4\}}]+[\widetilde{L}_{\{3,4\}}])(\mathbb{L}-1) \\ + ([\widetilde{L}_{\{1,2,4\}}]+[\widetilde{L}_{\{2,3,4\}}])(\mathbb{L}-1)^2.
\end{aligned}\end{equation}
Note that in the two motives of infinite cyclic covers we can now cancel 
$$(\ref{e3})\cdot (\mathbb{L}-1)+(\ref{e4})\cdot (\mathbb{L}-1)+(\ref{e34})\cdot (\mathbb{L}-1)^2.$$
This is a reflection of the inclusion-exclusion principle, showing that strata of the center $Z$ which are contained in the transversal components $E_3$, $E_4$ and their intersection $E_{\{3,4\}}$, give equal contributions to the two motives of infinite cyclic covers.

So it remains to show that the contribution of the dense open stratum of $Z$ to the motivic infinite cyclic cover coincides with the contribution of strata in the exceptional divisor $E_*$ which are not contained in the proper transforms of $E_3$ and $E_4$. That is, the invariance under blowup of the motivic infinite cyclic cover reduces to checking  that 
$$-[\widetilde{E}^\circ_{\{1,2\}}](\mathbb{L}-1)=[\widetilde{E}^\circ_\ast] - ([\widetilde{L}_{\{1\}}]+[\widetilde{L}_{\{2\}}])(\mathbb{L}-1).$$
Note that since $L_{\{1\}}$ is contained in the  intersection $E_\ast \cap E'_1$, we have that  $L_{\{1\}}$ is contained in the exceptional divisor of the blow up of $E_1$ along $E_1 \cap Z$, which is isomorphic to $E_1 \cap Z = E_1 \cap E_2$ (since the codimension of the center is one).  Thus we have that $E^\circ_{\{1,2\}}$ is isomorphic to $L_{\{1\}}$ and, after lifting this isomorphism to the corresponding covers of Definition \ref{ucov}, we obtain: $\widetilde{E}^\circ_{\{1,2\}} \cong \widetilde{L}_{\{1\}}$. Analogously, we have that  $\widetilde{E}^\circ_{\{1,2\}} \cong \widetilde{L}_{\{2\}}$. Furthermore, $E_\ast$ is a Zariski locally trivial fibration over $E_{\{1,2\}}$ with fibre $\mathbb{P}^1$. Therefore, when we restrict $E_*$ over $E^\circ_{\{1,2\}}$, we get a Zariski locally trivial fibration $E^\circ_\ast$ with fibre $\mathbb{C}^\ast$, because we delete the two different points in each fibre corresponding to the intersections with $L_{\{1\}}$ and $L_{\{2\}}$. Hence, as explained at the end of Example \ref{D}, we have a similar $\mathbb{C}^\ast$-fibration for the corresponding covering spaces, and the claim follows by multiplicativity of motives in a Zariski locally trivial fibration.

\hfill$\square$
\end{example}

\begin{proof} ({\it of Proposition \ref{ind}})\newline 
The proof is by induction on the dimension of the center of blowup. The beginning of induction (i.e., the case of one point) is proved in Prop. \ref{pt}. Note that, in general, the center of blowup $Z$ is either contained in a component $E_i$ of $E$, or it is transversal to it, or it doesn't intersect it at all. We refer to components of the second kind as transversal components of $E$ (with respect to $Z$). By collecting all indices $i$ of components of $E$ containing $Z$, we note that  the center $Z$ is contained in a set $E_I$ (for some $I \subseteq J$) given by intersections of components of the deleted divisor. In particular, $Z$ gets an induced stratification from that of $E_I$. So, there is a dense open stratum $Z \cap E^{\circ}_I$ in $Z$, together with positive codimension strata obtained by intersecting $Z$ with collections of transversal components.

We begin the proof by first studying the case when the center of blowup is of type $E_I$, for some $I \subseteq J$. 
Let $X'$ be the blowup of $X$ along the center $Z$ defined as the intersection $E_I:=\bigcap_{i=1}^k E_i$ of some of the irreducible components of the deleted divisor $E$, and also assume that the irreducible components $E_j$ for $j= k+1, \dots, \ell$ of $E$ intersect the center $Z$ transversally, and no other components of $E$ intersect  $Z$. In this case, $Z$ is stratified by a top dimensional open dense stratum $E^\circ_I$, and by positive codimension strata obtained by intersecting $Z$ with some of the transversal components $E_j$ (with $j= k+1, \dots, \ell$), i.e., strata of the form $E^\circ_{I \cup K}$, where $K \ne \emptyset$ and $K \subseteq \{k+1, \dots, \ell\}$. 
Therefore,  the contributions to the motivic infinite cyclic cover $S_{X,E}$ supported on $Z$ are
\begin{equation}\label{lhs}
(-1)^{k-1}[\widetilde{E^\circ_I}](\mathbb{L} -1)^{k-1} + \sum_{\emptyset \ne K \subseteq \{k+1, \dots, \ell\}}(-1)^{k+|K|-1}[\widetilde{E^\circ_{I \cup K}}](\mathbb{L} -1)^{k+ |K| -1}.
\end{equation}

After blowing up $X$ along $Z$, we get the deleted divisor $E' = (\bigcup_{j\in J} E'_j) \cup E_\ast$ of $X'=Bl_Z X$, where $E_\ast$ is the exceptional locus of the blowup and $E'_j$ is the proper transform of $E_j$ (for $j \in J$). Note that, by the choice of the center $Z$ of blowup, the $k$-fold intersection of the proper transforms of components $E_i$ with $i=1, \dots, k$ is empty,  i.e., $\bigcap_{i=1}^k E'_i = \emptyset$. The exceptional divisor $E_\ast$ is stratified by the top dimensional open stratum $L_{\emptyset}=E^\circ_\ast$, by the codimension $s$ (for  $s <k$) strata obtained by intersecting $E_*$  with $s$-fold intersections of the components $E'_1, \dots, E'_k$ of $E'$, i.e., by the strata   $L_G$ with $G \subset I$ a proper subset, and by  strata contained in intersections of the proper transforms $E'_j$ for $j = k+1, \dots, \ell$ of the transversal components, i.e., strata of the type  $L_{G \cup K}$ where $G \subset I$ is a proper subset of $I$ and $K$ is a nonempty subset of  $\{k+1, \dots, \ell\}$.  Therefore the contributions to the motivic infinite cyclic cover  $S_{X', E'}$ supported on $E_\ast$ are:
\begin{equation}\label{rhs}
 [\widetilde{E^\circ_\ast}] + \sum_{{\substack{G \subset I,\\ G \ne \emptyset, I}}} (-1)^{|G|}(\mathbb{L}-1)^{|G|} \Big ( [\widetilde{L_G}]+ \sum_{{\substack{K \subseteq \{k+1, \dots, \ell\} \\
 K\ne \emptyset}}} (-1)^{|K|}[\widetilde{L_{G \cup K}}] (\mathbb{L}-1)^{|K|}\Big ).\end{equation}

We can apply induction on the dimension of the center of blowup, and the exclusion-inclusion principle, to show that 
strata of the center $Z$ which are contained in intersections of the transversal components $E_j$, for $j=k+1,\cdots, \ell$, give equal contributions to the  motives $S_{X, E}$ and $S_{X', E'}$ of the corresponding infinite cyclic covers. More precisely,
for each positive codimension stratum $E^\circ_{I \cup K}$ of $Z$, we get by induction for the blowup of $E_K$ along the center $Z \cap E_K=E_{I \cup K}$, and with deletion divisor $E_K \cap (\sum_{i=1}^k E_i)$, a relation of the type
\begin{equation}\tag{$\ast_{K}$}\label{induc}
\begin{aligned}
(-1)^{|K|-1}[\widetilde{E^\circ_{I \cup K}}](\mathbb{L} -1)^{|K| -1} + \sum_{K \subset K' \subseteq \{k+1, \dots, \ell \}} (-1)^{|K'|-1} [\widetilde{E^{\circ}_{I \cup K'}}](\mathbb{L}-1)^{|K'|-1} \\
=[\widetilde{L_K}] + \sum_{{\substack{G \subset I,\\ G \ne \emptyset, I}}} \sum_{K \subseteq K' \subset \{k+1, \dots, \ell\}}(-1)^{|G \cup (K' \setminus K)|}[\widetilde{L_{G \cup K'}}](\mathbb{L}-1)^{|G \cup (K' \setminus K)|}.\end{aligned}\end{equation}
By summing up all the products (\ref{induc})$\cdot (\mathbb{L} -1)^{|K|}$ for the positive codimension strata $E^\circ_{I \cup K}$ of $Z$, we reduce the comparison of (\ref{lhs}) and (\ref{rhs}) to proving the identity:
\begin{equation}\label{finaleq}
(-1)^{k-1}[\widetilde{E^\circ_I}](\mathbb{L} -1)^{k-1} =  [\widetilde{E^\circ_\ast}] + \sum_{{\substack{G \subset I,\\ G \ne \emptyset, I}}} (-1)^{|G|}[\widetilde{L_G}](\mathbb{L}-1)^{|G|},
\end{equation}
i.e.,  it remains to show that the contribution of the dense open stratum of the center $Z$ to the motivic infinite cyclic cover $S_{X, E}$ coincides with the contribution to $S_{X', E'}$ of any of the strata supported on the exceptional divisor $E_*$ which are not contained in the proper transforms of the components of $E$ which are transversal to $Z$. 

Note that, for any subset $G \subsetneq I=\{ 1, \dots, k\}$ (including the empty set corresponding to  $L_\emptyset = E^\circ_\ast$), we have that $L_G$ is a Zariski locally trivial fibration over $E_I^{\circ}$ with fibre $(\mathbb{C}^\ast)^{k - |G| - 1}$. Indeed, the closure $\bar{L}_G$ of $L_G$ is the exceptional divisor of the blowup of $E_G$ along $Z$. Therefore  $\bar{L}_G$ is a Zariski locally trivial fibration over $Z$ with fibre isomorphic to $\mathbb{P}^{k -|G| - 1}$. When we restrict the fibration $\bar{L}_G \rightarrow Z$ over the open dense stratum $E^\circ_I$ of $Z$, we remove the fibers lying above the intersections of $Z$ with the transversal components $E_{k+1}, \dots, E_{\ell}$. To obtain $L_G$, we need to further subtract the intersections of the total space of the fibration $(\bar{L}_G)_{|_{E^\circ_I}}$ with the components $E'_i$ with $i \in I \setminus G$. Fiberwise, the effect of the latter operation is that we remove $k-|G|$ hyperplanes in general position, hence the fiber of $L_G \to E_I^{\circ}$ is isomorphic to a complex torus $(\mathbb{C}^\ast)^{k - |G| - 1}$ of dimension $k-|G|-1$. 

By Lemma \ref{techlem}, the (Zariski) locally trivial fibration $L_G \to E_I^{\circ}$ with fiber isomorphic to $(\mathbb{C}^\ast)^{k - |G| - 1}$ can be lifted to a $(\mathbb{C}^\ast)^{k - |G| - 1}$-fibration $\widetilde{L_G} \to \widetilde{E^{\circ}_I}$. Thus, the Zariski triviality implies that $$[\widetilde{L_G}] = [ \widetilde{E^{\circ}_I}] (\LL-1)^{k - |G| - 1}.$$
Finally, the equality (\ref{finaleq}) follows from the Pascal triangle because the number of subsets $G$ of $I$ of given size  $|G|$ equals the binomial coefficient ${ k \choose |G|}$.  

\bigskip

Let us now explain the proof in the general case, i.e., when  the center $Z$ is strictly contained in some set $E_I$, for $I \subseteq J$, and let $I=\{1,\cdots,k\}$. Assume that the codimension of $Z$ in $X$ is $r+1\geq k$. Again, by induction, it suffices to show   that the contribution of the dense open stratum $Z^{\circ}:=Z \cap E_I^{\circ}$ of the center $Z$ to the motivic infinite cyclic cover $S_{X, E}$ coincides with the contribution to $S_{X', E'}$ of the strata supported on the exceptional divisor $E_*=\mathbb{P}(\nu_Z)$ which are not contained in the proper transforms of the transversal components components of $E$ (with respect to $Z$), that is, 
\begin{equation}\label{finaleq2}
(-1)^{k-1}[\widetilde{Z}^\circ](\mathbb{L} -1)^{k-1} =  [\widetilde{E^\circ_\ast}] + \sum_{{\substack{G \subseteq I,\\ G \ne \emptyset}}} (-1)^{|G|}[\widetilde{L_G}](\mathbb{L}-1)^{|G|}.
\end{equation}
On the right hand side of (\ref{finaleq2}), we use the same notation as before for the stratification of the exceptional divisor $E_*$. Note that in this case we have to also allow $G=I$ in the sum of the right hand side term of (\ref{finaleq2}) because $Z \subsetneq E_I$ and therefore $\bigcap_{i=1}^k E'_i \neq \emptyset$. 

Note that, for any subset $G \subsetneq I=\{ 1, \dots, k\}$ (including the empty set corresponding to  $L_\emptyset = E^\circ_\ast$), we have that $L_G$ is a Zariski locally trivial fibration over $Z^{\circ}$ with fibre $\mathbb{C}^{r-k+1} \times (\mathbb{C}^\ast)^{k - |G| - 1}$. Indeed, the closure $\bar{L}_G$ of $L_G$ is the exceptional divisor of the blowup of $E_G$ along $Z$. Therefore  $\bar{L}_G$ is a Zariski locally trivial fibration over $Z$ with fibre isomorphic to $\mathbb{P}^{r -|G|}$. When we restrict the fibration $\bar{L}_G \rightarrow Z$ over the open dense stratum $Z^\circ$ of $Z$, we remove the fibers lying above the intersections of $Z$ with the transversal components $E_{k+1}, \dots, E_{\ell}$. To obtain $L_G$, we need to further subtract the intersections of the total space of the fibration $(\bar{L}_G)_{|_{Z^\circ}}$ with the components $E'_i$ with $i \in I \setminus G$. Fiberwise, the effect of the latter operation is that we remove $k-|G|$ hyperplanes in general position from $\mathbb{P}^{r -|G|}$, hence the fiber of $L_G \to Z^{\circ}$ is isomorphic to the cartesian product  $\mathbb{C}^{r+1-k} \times (\mathbb{C}^\ast)^{k - |G| - 1}$ of a complex affine space of dimension $r+1-k$ and a complex torus of dimension $k-|G|-1$. In the case $G=I$, we get that ${L}_I$ is a Zariski locally trivial fibration over  ${Z}^\circ$ with fibre the projective space $\mathbb{P}^{r-k}$.

By Lemma \ref{techlem}, the (Zariski) locally trivial fibration $L_G \to Z^{\circ}$ with fiber isomorphic to $\mathbb{C}^{r-k+1} \times (\mathbb{C}^\ast)^{k - |G| - 1}$ can be lifted to a $\mathbb{C}^{r-k+1} \times (\mathbb{C}^\ast)^{k - |G| - 1}$-fibration $\widetilde{L_G} \to \widetilde{Z^{\circ}}$. Thus, the Zariski triviality implies that, for $G \subsetneq I=\{ 1, \dots, k\}$ (including the empty set corresponding to  $L_\emptyset = E^\circ_\ast$), we have: 
\begin{equation}\label{20a} [\widetilde{L_G}] = [ \widetilde{Z^{\circ}}] \LL^{r-k+1}(\LL -1)^{k - |G| - 1}.\end{equation}
For $G=I$, the fiber $\mathbb{P}^{r-k}$ of $L_I \to Z^{\circ}$ is simply connected (as $r-k\geq 1$),  hence the covering $L_I \to Z^{\circ}$ can be lifted to a $\mathbb{P}^{r-k}$-fibration $\widetilde{L_I} \to \widetilde{Z^{\circ}}$. Thus, Zariski locally triviality yields that 
\begin{equation}\label{20b} [\widetilde{L_I}] = [\widetilde{Z^\circ}](\LL^{r-k} + \LL^{r-k-1}+\cdots + \LL + 1).\end{equation}
By substituting the equalities (\ref{20a}) and (\ref{20b}) into  (\ref{finaleq2}), and factoring out $[\widetilde{Z}^\circ](\LL-1)^{k-1}$, it remains to show that:
\begin{equation}\label{finaleq3}(-1)^{k-1} =  \LL^{r-k+1} + \sum_{{\substack{G \subsetneq I,\\ G \ne \emptyset}}} (-1)^{|G|}\LL^{r-k+1}+(-1)^{k}(\LL^{r-k} + \LL^{r-k-1}+ \cdots + \LL + 1)(\LL-1).\end{equation}
Note that the right hand side of (\ref{finaleq3}) can be written as:
$$\left[ \sum_{G \subseteq I} (-1)^{|G|} \LL^{r-k+1} \right] + (-1)^{k-1}.$$
So after cancelling $(-1)^{k-1}$ from both sides of  (\ref{finaleq3}), it remains to show that:
\begin{equation}\label{20c}
 \sum_{G \subseteq I} (-1)^{|G|}=0.
\end{equation}
Since the number of subsets $G$ of $I$ of given size $i$ equals the binomial coefficient ${k \choose i}$, 
it follows that (\ref{20c}) is equivalent to the following well-known identity:
$$\sum_{i=0}^{k} (-1)^i { k \choose i} = 0.$$ Thus equation  (\ref{finaleq2}) holds.
\end{proof}

%%%%%%%%%%%%%%%%%%%%%%%%%%%%%%%%%%%%
%%%%%%%%%%%%%%%%%%%%%%%%%%%%%%%%%%%%

\section{Betti Realization}\label{s4}

Let $V_\mathbb{Q}^{{\rm end}}$ be the category of 
finite dimensional 
$\mathbb{Q}$-vector spaces endowed with an endomorphism.  Remark that $V_\mathbb{Q}^{{\rm end}}$ is equivalent to the category  
of torsion $\mathbb{Q}[t]$-modules, e.g., see \cite[Section 3]{KS}. There exists a $\mathbb{Q}$-linear  homomorphism
$$\xi : V_\mathbb{Q}^{{\rm end}} \rightarrow \mathbb{Q}(t)$$ 
defined by $$(V, M) \mapsto {\rm exp}\left(\sum_n \frac{{\rm Trace}(M^n)}{n} t^n\right)={1 \over {\det(Id-t M)}},$$
which satisfies $$\xi\left((V,M)\right) = \xi\left((V_1,M_1)\right) \cdot \xi\left((V_2,M_2)\right)$$
for each exact sequence $0 \rightarrow V_1 \rightarrow V 
\rightarrow V_2\rightarrow 0$ such that $V_1$ is $M$-invariant, 
$M\vert_{V_1}=M_1$ and the map induced by $M$ on $V_2=V/V_1$ 
conicides with $M_2$.
\begin{remark}\label{lr}
Note that if $M^s$ denotes the semi-simple part of the endomorphism $M$, then $\xi\left((V,M)\right)=\xi\left((V,M^s)\right)$. So for the definition of $\xi$ it suffices to take into consideration only the semi-simple part of $M$.
\end{remark}

By \cite[p.377]{A}, there is a monomorphism
$$(\xi, {\textit for}):K_0(V_{\QQ}^{\rm end}) \rightarrow \QQ(t)^* \times K_0(V_{\QQ}),$$
with $V_{\QQ}$ the abelian category of finite dimensional rational vector spaces, and $${\textit for}:K_0(V_{\QQ}^{\rm end}) \rightarrow K_0(V_{\QQ}), \ \ \left[(V,M)\right] \mapsto [V]$$ the corresponding forgetful functor. 
Hence, by Remark \ref{lr}, this identification yields that \begin{equation}[(V,M)]=[(V,M^s)] \in K_0(V_{\QQ}^{\rm end}).\end{equation}

Denote by $V_\mathbb{Q}^{{\rm aut}}$ the category of 
finite dimensional 
$\mathbb{Q}$-vector spaces endowed with a finite order automorphism. 
Then there exists an additive map (called the {\it Betti realization}) 
$$\chi_{b} : K_0({\rm Var}^{\hat{\mu}}_{\CC}) \rightarrow K_0(V_\mathbb{Q}^{{\rm aut}}) \rightarrow K_0(V_\mathbb{Q}^{{\rm end}})$$  such that 
$$[Y, \sigma] \mapsto [H^{*}_c (Y; \mathbb{Q}), \sigma^*]:= 
\sum_{i\geq 0} (-1)^i [H^i_c (Y; \mathbb{Q}), \sigma^*_i],$$
where $\sigma^*_i$ denotes the automorphism of $H^i_c (Y, \mathbb{Q})$ induced by the action of $\sigma$. Here, the compactly supported cohomology is used in order to fit with the scissor relation (\ref{sci}) in the motivic Grothendieck group $K_0({\rm Var}^{\hat{\mu}}_{\CC})$.

We can therefore define a homomorphism
$$\xi_{mot} : K_0({\rm Var}^{\hat{\mu}}_{\CC}) \rightarrow (\QQ(t)^*, \cdot)$$
by the composition $\xi_{mot}:=\xi \circ \chi_b$, i.e.,
$$\xi_{mot} ([Y, \sigma])=Z_Y(t),$$
where $$Z_Y(t):=\prod_{i\geq 0} \left[ \det({\rm Id} -t\cdot \sigma^*_i, H_c^i(Y;\QQ)) \right]^{(-1)^{i+1}}$$ is the zeta function of the $\hat{\mu}$-action $\sigma$ on $Y$.

\bigskip

Back to our geometric situation, the deck transformation $T$ of the infinite cyclic cover $\widetilde{T}^*_{X,E,\Delta}$ induces  automorphisms $T^*_i$ on each group $H^i_c(\widetilde{T}^*_{X,E,\Delta})$. The corresponding zeta function of $T$ is defined by: $$Z_{\widetilde{T}^*_{X,E,\Delta}}(t):=\prod_{i\geq 0} \left[ \det\left({\rm Id} -t\cdot T_i^*, H_c^i(\widetilde{T}^*_{X,E,\Delta})\right) \right]^{(-1)^{i+1}}.$$ Recall from Remark \ref{lr} that it suffices to take into consideration only the semisimple (hence of finite order) part of $T^*_i$.

The main result of this section describes the Betti realization of the motivic infinite cyclic cover $S_{X, E,\Delta}$.

\begin{prop}\label{bettirea}
The Betti realization of the motivic infinite cyclic cover 
of finite type is given  by the cohomology with compact support of $\widetilde{T}^*_{X,E,\Delta}$, i.e., 
\begin{equation}\label{br} \chi_{b}(S_{X, E,\Delta})=\sum_{i\geq 0} (-1)^i [H^i_c(\widetilde{T}^*_{X,E,\Delta}), T^*_i] \in K_0(V_\mathbb{Q}^{{\rm aut}}).\end{equation}
Equivalently, \begin{equation}\label{z}\xi_{mot} (S_{X, E,\Delta})=Z_{\widetilde{T}^*_{X,E,\Delta}}(t).\end{equation}
In particular, (by taking degrees) the topological Euler characteristic of $\widetilde{T}^*_{X,E,\Delta}$ is computed by \begin{equation}\label{eu}\chi(\widetilde{T}^*_{X,E,\Delta})=\chi(S_{X, E,\Delta}).\end{equation}
\end{prop}

\begin{proof} 

Consider the $T$-equivariant Mayer-Vietoris spectral sequence for the open cover $\{T^*_{E^\circ_i}\}_{i \in J}$ of ${T}^*_{X,E,\Delta}$, i.e., 
$$E_1^{p,q}=\bigoplus_{|I|=p+1}H^q_c(\widetilde{T}^*_{E_I^{\circ}})
\Rightarrow H_c^{p+q}(\widetilde{T}_{X,E,\Delta}^*),$$
in which we identify the intersections
$\bigcap_{i \in I} T^*_{E^\circ_i}=T^*_{E^\circ_I}$ as in Proposition \ref{ft}.
By using the additivity of the universal Euler characteristic $W \rightarrow [W] \in
K(V^{\rm aut}_{\QQ})$,
for $W\in V^{\rm aut}_{\QQ}$, we have the identity:
\[
\begin{aligned}
\sum_{i\geq 0} (-1)^i[ H_c^{i}(\widetilde{T}_{X,E,\Delta}^*)] &=\sum_{i,j \geq 0} (-1)^{i+j} [E_1^{i,j}] \\
&= \sum_{i,j\geq 0} (-1)^{i+j}\left[\bigoplus_{|I|=j+1}H^i_c(\widetilde{T}^*_{E_I^{\circ}})\right] \\
&= \sum_{\emptyset \not = I \subseteq J} (-1)^{|I|-1} \left( \sum_{i \geq 0} (-1)^i[H^i_c(\widetilde{T}^*_{E_I^{\circ}})] \right). 
\end{aligned}
\]
By using the definition of the motivic infinite cyclic cover $S_{X, E,\Delta}$, it thus suffices  to check formula (\ref{br}) for each stratum $E_I^{\circ}$, i.e., we have to show that the following identity holds
\begin{equation}\label{ad}  \chi_{b}([\widetilde{E}^\circ_I](\LL-1)^{|I| -1})=\sum_{i \geq 0} (-1)^i[H^i_c(\widetilde{T}^*_{E^\circ_I})],\end{equation}
or equivalently (after applying $\xi$),
\begin{equation}\label{suf}
\begin{aligned}
\xi_{mot}([\widetilde{E}^\circ_I](\LL-1)^{|I| -1}, \sigma_{I})&=\prod_{i\geq 0} \left[ \det({\rm Id} -t\cdot T_i^*, H_c^i(\widetilde{T}^*_{E_I^\circ})) \right]^{(-1)^{i+1}}\\&=Z_{\widetilde{T}^*_{E_I^\circ}}(t), 
\end{aligned}\end{equation}
 with $\sigma_I$ denoting the corresponding $ \mu_{m_I}$-action.
On the other hand, by definition, 
$$\xi_{mot}([\widetilde{E}^\circ_I](\LL-1)^{|I| -1}, \sigma_{I})=Z_{\widetilde{E}^\circ_I \times (\CC^*)^{|I|-1}}(t),$$
so it remains to prove the equality of zeta functions:
\begin{equation}\label{suff}
Z_{\widetilde{T}^*_{E_I^\circ}}(t)=Z_{\widetilde{E}^\circ_I \times (\CC^*)^{|I|-1}}(t).
\end{equation}
(Note that in (\ref{suff}), the product $\widetilde{E}^\circ_I \times (\CC^*)^{|I|-1}$ can be replaced by any Zariski locally trivial fibration over $\widetilde{E}^\circ_I $ with fiber $(\CC^*)^{|I|-1}$, as they give the same element in $K_0({\rm Var}^{\hat{\mu}}_{\CC})$.)

By Lemmas \ref{techlem} and  \ref{definingcovers}, the long exact sequence of homotopy groups associated to the $(\CC^{\ast})^{|I|}$-fibration $T^\ast_{E^\circ_I} \to {E^\circ_I}$ induces a locally trivial topological fibration 
\begin{equation}\label{fibr}\widetilde{T}^\ast_{E^\circ_I} \rightarrow \widetilde{E}^\circ_I,\end{equation}
with connected fiber $\widetilde{(\mathbb{C}^\ast)^{|I|}} \simeq (\CC^*)^{|I|-1}$, the infinite cyclic cover of $(\CC^{\ast})^{|I|}$ defined by the kernel of the epimorphism $\ZZ^{|I|} \twoheadrightarrow m_I \ZZ$ induced by the holonomy map $\Delta$. As before, 
$\widetilde{T}^\ast_{E^\circ_I}$ is the infinite cyclic cover of $T^\ast_{E^\circ_I}$ defined by $\Delta$, and $\widetilde{E}^\circ_I$ is the unbranched $\mu_{m_I}$-cover of $E^\circ_I$ with holonomy $\Delta_I$. 

For a sufficiently fine cover of $\widetilde{E}^\circ_I$ by
$\mu_{m_I}$-invariant sets, the fibration (\ref{fibr}) becomes 
trivial. Hence (\ref{suff}) follows from the multiplicativity of 
zeta functions, i.e., from the equality $$Z_{U_1\cup U_2}(t)=
{{Z_{U_1}(t) \cdot Z_{U_2}(t)} \over {Z_{U_0}(t)}}$$ for $T$-invariant sets 
$U_1,U_2$ with $U_0=U_1 \cap U_2$. This multiplicativity is easily deduced from the corresponding Mayer-Vietoris long exact sequence.

It is now easy to see that the common value of the terms in (\ref{suff}) is $(1-t^{m_I})^{-\chi(E_I^\circ)}$ if $|I|=1$, and it is $1$ otherwise.
\end{proof}

\begin{remark} More generally, there is a {\it Hodge realization} homomorphism $$\chi_h:K_0({\rm Var}^{\hat{\mu}}_{\CC}) \rightarrow K_0({\rm HS}^{{\rm mon}})$$ defined by the same formula as $\chi_b$, 
with $K_0({\rm HS}^{{\rm mon}})$ the Grothendieck group of monodromic Hodge structures (i.e., endowed with an automorphism of finite order), cf. \cite{Barcelona}. In the case when the compactly supported cohomology of the infinite cyclic cover of $T^\ast_{E}$ admits mixed Hodge structures (e.g., see \cite[Section 6]{LM} for such a situation), the above proposition can be extended to show that the corresponding class in $K_0({\rm HS}^{{\rm mon}})$ is the Hodge realization of the motivic infinite cyclic cover of $T^\ast_{E}$.
On the other hand, for an arbitrary infinite cyclic cover $\widetilde T^*_{X,E,\Delta}$ of finite type, 
even when a construction of a mixed Hodge structure is absent, $\chi_h$ 
provides ``its class''.
\end{remark}

%%%%%%%%%%%%%%%%%%%%%%%%%%%%
%%%%%%%%%%%%%%%%%%%%%%%%%%%%

\section{Relation with motivic Milnor fiber}\label{relationwithmotivic}

Denef and Loeser introduced the {\it local motivic Milnor fibre} $\mathcal{S}_{f,x}$ at a point $x$ for a non-constant morphism $f: \mathbb{C}^{d+1} \rightarrow \mathbb{C}$  with $f(x)=0$ (e.g., see \cite[Def.3.2.1, Def.3.5.3]{Barcelona} and the references therein) as a limit in the sense of \cite[Section 2.8]{GLM} (see \cite[Lemma 4.1.1]{DenefLoeserJAG}): 
\begin{equation}\label{bla} \mathcal{S}_{f,x}: = - \lim_{T \to + \infty} Z(T) \in K_0({\rm Var}^{\hat{\mu}}_\mathbb{C})[\mathbb{L}^{-1}] \end{equation}
 of the {\it motivic zeta function}
\begin{equation}\label{blaa}
Z(T):=\sum_{n\geq 1} [\mathcal{X}_{n,1}]\LL^{-(d+1)n}T^n \in K_0({\rm Var}^{\hat{\mu}}_\mathbb{C})[\mathbb{L}^{-1}][[T]],\end{equation} where $\mathcal{X}_{n,1}$ denotes the set of $(n+1)$-jets $\varphi$ of $ \mathbb{C}^{d+1}$  centered at $x$ such that $f \circ \varphi = t^n + \dots$. 
Note that there is a good action of the group $\mu_n$ (hence of $\hat\mu$) on $\mathcal{X}_{n,1}$ by $\lambda \times \varphi \mapsto \varphi( \lambda \cdot t)$.

\medskip

The following result relates the concepts of motivic Milnor fiber and the motivic infinite cyclic cover, respectively. 
\begin{theorem}\label{milnorfiberinfcover}
Let $f: \mathbb{C}^{d+1} \rightarrow \mathbb{C}$ be a 
non-constant morphism with $f(x)=0$,  and 
let $p : X \rightarrow \mathbb{C}^{d+1}$ 
be a log-resolution of the singularities of pair 
$(\CC^{d+1}, f^{-1}(0))$. 
Choose $p$ in such a way that $(p^{-1}(x))_{{\rm red}}$ is a union of components of $(p^{-1}(f^{-1}(0)))_{{\rm red}}$. Let $E = \sum_{j \in J} E_j$ be the irreducible component decomposition  of $p^{-1}(f^{-1}(0))_{{\rm red}}$, and  let $A = \{i \in J \, | \, E_i \subset p^{-1}(x)\}$. 
Then the following hold:
\begin{enumerate}
\item[(1)] For $\epsilon > 0$ small enough, and $B(x,\epsilon)$ a ball of radius $\epsilon$ centered at $x \in \CC^{d+1}$, the map $p$ provides a biholomorpic identification between  
$B(x,\epsilon)\setminus \{f=0\}$ and $T^\ast_{E^A}$, the punctured regular neighborhood of the divisor $E^A:=\sum_{i \in A} E_i$. 
In particular, the map $\gamma \rightarrow \int_{\gamma} {df \over f}$ 
can be viewed as a holonomy homomorphism: 
$\Delta: \pi_1(T^\ast_{E^A}) \rightarrow \mathbb{Z}$  
of the punctured neighborhood of $E^A$. This holonomy map   
takes the boundary $\delta_i$ of any small disk transversal 
to the irreducible component $E_i$ of $E^A$ to the multiplicity $m_i$ of $E_i$ in the divisor of $f \circ p$, i.e., $\Delta(\delta_i)=m_i$ for all $i \in A$. 

\item[(2)] One has the following identity in $K_0({\rm Var}^{\hat{\mu}}_\mathbb{C})[\mathbb{L}^{-1}]$:
$$\mathcal{S}_{f,x} = S^A_{X,E,\Delta}.$$ 
\end{enumerate}
\end{theorem}

\begin{proof} We shall use the following formula
(e.g., \cite[Def.3.5.3]{Barcelona}, which in turn is motivated by the calculation in \cite[Theorem 2.2.1]{DenefLoeserJAG} 
and \cite[Theorem 2.4]{DenefLoeserTopology}) 
for the motivic Milnor fiber $\mathcal{S}_{f,x}$ 
in terms of a log-resolution $p: X \rightarrow \mathbb{C}^{d+1}$ of $f^{-1}(0)$:
\begin{equation}\label{35} \mathcal{S}_{f,x}= \sum_{\substack{\emptyset \neq I \subseteq J \\ I \cap A \ne \emptyset}} (-1)^{|I|-1} c_I (\LL-1)^{|I|-1},\end{equation}
where 
$c_I $ 
is the class of the unramified Galois cover $\overset{\approx}{E^\circ_I} \in {\rm Var}_{\mathbb{C}}^{\hat \mu}$ of $E^\circ_I$, 
with Galois group $\mu_{m_I}$, defined as follows. 
Let $m_i$ be the multiplicity of $E_i$ in the divisor of $f \circ p$
and $m_I=\gcd(m_i | i \in I)$.
Given an affine Zariski open subset $U$ of $X$ such 
that $f \circ p = uv^{m_I}$ on $U$, with $u \in {\Gamma}(U,\mathcal{O}_U)$  
a unit and $v$ 
a morphism from $U$ to $\mathbb{C}$, the restriction $\overset{\approx}{E^\circ_I}|_U$ of $\overset{\approx}{E^\circ_I}$ over ${E^\circ_I}|_U:={E^\circ_I} \cap U$ is defined by
\begin{equation}\overset{\approx}{E^\circ_I}|_U=
\{(z,y) \in \mathbb{C} \times E^\circ_I\vert_U \, | \, 
z^{m_I}=cu^{-1}\}.\end{equation}
There is a natural $\mu_{m_I}$-action defined by multiplying the 
$z$-coordinate with the elements of $\mu_{m_I}$, whose corresponding quotient 
yields the covering map: $\overset{\approx}{E^\circ_I}|_U \rightarrow {E^\circ_I}\vert_U$. 
We denote this action $\sigma_I'$.
For proving our theorem, 
it suffices to  
 show that, as elements of ${\rm Var}^{\hat{\mu}}_\mathbb{C}$, the cover $(\overset{\approx}{E^\circ_I}|_U,\sigma_I')$
coincides with the cover 
$(\widetilde{E}^\circ_I\vert_U,\sigma_I)$ from Definitions 
\ref{ucov} and \ref{maindef}, where we let as above $\widetilde{E}^\circ_I\vert_U$ denote the restriction of $\widetilde{E}^\circ_I$ over $U$.

Let $M_f$ denote the Milnor fiber $\{f=c\} \cap B(x,\epsilon)
\subset B(x,\epsilon)\setminus \{f=0\} \cong T^*_{E^A}$. For a sufficiently 
small subset $U \subset E_I^\circ$ we can choose a trivialization 
of $T_{E^A}\vert_U$ which yields a trivialization as a $({\CC^*})^{|I|}$-bundle
of the subset $T^*_{U,E_I^\circ}:=T_{E^\circ_I}^*\vert_U$
of $T_{E^A}^*$.
Let 
$$M_{U,E_I^{\circ}}=
M_f \cap T^*_{U,E_I^{\circ}} \subset T^*_{U,E_I^{\circ}}=
E_I^\circ\vert_U \times (\CC^*)^r,$$
with $r=|I|$.
In the latter identification, $M_{U,E_I^{\circ}}$ is the hypersurface
given by $z_1^{m_1}\cdots z_r^{m_r}=cu^{-1}$ (where $z_i$ are the 
coordinates in the torus). It follows that fibers of $M_{U,E_I^{\circ}}$ 
over $E_I^{\circ}\vert_U$ are disjoint unions of $m_I$ translated subgroups
$z_1^{m_1 \over m_I}\cdots z_r^{m_r \over m_I}=\lambda \omega_{m_I}$ where
$\lambda^{m_I}=cu^{-1}$ and $\omega_{m_I}\in \mu_{m_I}$.
Each such translated subgroup is biholomorphic to a torus
$({\CC^*})^{r-1}$. In fact, the Stein factorization presents 
$M_{U,E_I^{\circ}}$ as a $(\CC^*)^{r-1}$-torus fibration over 
$\overset{\approx}{E^\circ_I}|_U$, with the map $M_{U,E_I^{\circ}} \rightarrow
\overset{\approx}{E^\circ_I}|_U$ induced by $(z_1,...z_r) \mapsto 
z=z_1^{m_1 \over m_I}\cdots z_r^{m_r \over m_I}$.

Next consider the following commutative diagram:
$$\begin{matrix}  \pi_1((\CC^*)^{r-1})& \rightarrow &
     \pi_1((\CC^*)^r) &  & \cr
             \downarrow  & & \downarrow &  \searrow & \cr
           \pi_1(M_{U,E_I^\circ}) & \rightarrow 
& \pi_1(T^*_{U,E_I^\circ})& \buildrel \Delta \over \rightarrow & \ZZ\cr
 \downarrow  & & \downarrow & & \downarrow \cr
\pi_1(\overset{\approx}{E^\circ_I}|_U)& \rightarrow & \pi_1({E_I^\circ}\vert_U)&
\buildrel \Delta_{m_I} \over 
\rightarrow & \ZZ_{m_I}\cr
\end{matrix}$$
induced by the fibrations described above. Here $\Delta$ is the 
holonomy  described in (1) and $\Delta_{m_I}$ is the map 
induced by $\Delta$ as described in  Lemma \ref{techlem}.

To conclude the proof of the 
theorem it is enough to show that image of $\pi_1(\overset{\approx}{E^\circ_I}|_U)$
belongs to the kernel of the map $\Delta_{m_I}$, 
since $\ker(\Delta_{m_I})=\pi_1(\widetilde{E}_I^\circ\vert_U)$
by our construction and both groups $\pi_1(\widetilde{E}_I^\circ\vert_U)$
and $\pi_1(\overset{\approx}{E^\circ_I}|_U)$ have index $m_I$ in 
$\pi_1({E_I^\circ}\vert_U)$.
Notice that the restriction of $\Delta$ on $M_f \subset T_{E^A}^*$ 
yields $\Delta: \pi_1(M_f) \rightarrow \ZZ$, which is trivial 
since for any $\gamma \subset M_f$ one has $\int_{\gamma} {{df} \over f}=0$
(as $f(\gamma)$ is constant). Hence the composition of maps in the middle row of the above diagram is trivial. By commutativity, the image of the composition
$$\pi_1(M_{U,E_I^{\circ}}) \rightarrow \pi_1(\overset{\approx}{E^\circ_I}|_U) \rightarrow  \pi_1({E_I^\circ}\vert_U) \rightarrow \ZZ_{m_I}$$
is also trivial.
Moreover, the homomorphism  $\pi_1(M_{U,E_I^\circ}) \rightarrow 
 \pi_1(\overset{\approx}{E^\circ_I}|_U)$ is surjective since
it is induced by the map $M_{U,E_I^\circ} \rightarrow 
\overset{\approx}{E^\circ_I}|_U$ which is a fibration with connected fibers.
Therefore the composition  
$\pi_1(\overset{\approx}{E^\circ_I}|_U) \rightarrow \pi_1({E_I^\circ}\vert_U)
\rightarrow \ZZ_{m_I}$ is trivial and 
the claim follows.
\end{proof}

\begin{remark} Note that Theorems \ref{milnorfiberinfcover} and \ref{Welldefined} give a direct proof of the fact that the right-hand side of formula (\ref{35}) expressing the Denef-Loeser motivic Milnor fiber in terms of a log-resolution is actually independent of the choice of log-resolution. This was apriori known only because of the relation (\ref{bla}) with the motivic zeta function (which is intrinsically defined by Denef-Loeser in terms of arc spaces as in (\ref{blaa})), see also the discussion in \cite[Section 3.5]{Barcelona}. It should also be noted that our proof of independence of (\ref{35}) of the choice of log resolution does not make sense of the third relation (\ref{r3}) in the motivic Grotehendieck group $K_0({\rm Var}^{\hat{\mu}}_\mathbb{C})$. As a consequence, our results also imply the well-definedness of the Denef-Loeser motivic nearby and vanishing cycles without the use of the third relation (\ref{r3}) in the motivic Grothendieck group (which was needed for the approach via arc spaces).
\end{remark}

Let us consider now a non-constant morphism $f : \CC^{d+1} \rightarrow \CC$ with $f(0)=0$. As described at the end of Section 2, by Milnor's fibration theorem \cite{Milnor}, there is a locally trivial fibration $\pi : B_{\epsilon, \delta} \rightarrow D^*_\delta$ associated to $f$ and the origin $0 \in \CC^n$. Let us call $T_f$ the corresponding monodromy map.  Since the infinite cyclic cover of $B_{\epsilon, \delta}$ and the Milnor fiber $M_f$ at the origin are homotopically equivalent, we have the following corollary as  a consequence of Theorems \ref{bettirea} and \ref{milnorfiberinfcover}. This is a weak version of Theorem 4.2.1 in  \cite{DenefLoeserJAG}, see also \cite{DenefLoeserTopology}.

\begin{cor}\label{bettireapol}The Betti realization of the motivic Milnor fiber 
of  $f$ at the origin coincides with the Betti invariant of the monodromy $T_{f}$, i.e.,  $$\chi_b(\mathcal{S}_{f,0}) = \sum_i (-1)^i[H^i_c(M_{f}),
T_{M_{f}}^*].$$\end{cor}

%%%%%%%%%%%%%%%%%%%%%%%%%%%%%%%%%%%%%%%%%%

\section{Motivic Milnor fibers at infinity and motivic 
milnor fibers associated with rational functions}\label{s6}

\bigskip

In this last section, we outline further geometric situations
in which our main construction allows to obtain motivic invariants
for which we also obtain Betti realizations.

\smallskip

Let $f,g \in \CC[x_1,\cdots, x_n]$ be two polynomials, with 
$\deg(f)-\deg(g)=k \ge 0$. Consider the pencil of hypersurfaces
in $\PP^n={\rm Proj} \ \CC[x_0,x_1,\cdots ,x_n]$ 
of degree $\deg(f)$ given by $$\lambda\bar f+\mu \bar g x_0^k=0,$$
where $\bar f, \bar g$ denote the homogenizations of $f$ and $g$, respectively, and $[\lambda:\mu] \in \PP^1$.  
The rational map $\pi_{f,g}: \PP^n \rightarrow \PP^1$ corresponding 
to this pencil is given by 
$[x_0 : \dots : x_n] \mapsto [ \bar f : \bar gx_0^k]$.
Let 
$\phi : \widetilde {\PP^n}_{f,g} \rightarrow \PP^n$ 
be a resolution 
of the indeterminacy points of the rational map $\pi_{f,g}$ (i.e., the set of solutions of 
$\bar f=\bar gx_0^k=0$), cf. \cite[7.1.2]{Dolgachev}; in a small ball about an indeterminacy point  the restriction of the map  $\pi_{f,g}$ to the complement of $\{g=0\}$ is given by $f/g$, where the target of $f/g$ is identified with $\PP^1 \setminus \{ [1:0] \}$.   Let us denote by $\widetilde \pi_{f,g}$ the composition
$\widetilde {\PP^n}_{f,g}\buildrel {\phi} 
\over  \rightarrow \PP^n \buildrel {\pi_{f,g}} 
\over \rightarrow \PP^1$. Note that the proper transforms under $\phi$ of divisors
$\pi^{-1}_{f,g}([\lambda : \mu])$  and 
$\pi^{-1}_{f,g}([\lambda' : \mu'])$
have empty intersection provided 
 $[\lambda : \mu] \ne 
[\lambda' : \mu']$.
After possibly additional blow-ups, we can assume 
(using the same notations) that 
the fibers $E_{0}= \widetilde \pi_{f,g}^{-1}([0:1])$ and  $E_{\infty}=\widetilde \pi_{f,g}^{-1}([1 : 0])$ 
(i.e., the total transform of the divisors $\bar{f}=0$ and  $\bar{g}x_0^k=0$, respectively)  are  both normal 
crossing divisor on $\widetilde {\PP^n}_{f,g}$. We shall assume from now 
on that $\widetilde {\PP^n}_{f,g}$ already satisfies this condition.

\smallskip

The following is a standard consequence of transversality 
theory in the context of stratified spaces.

\begin{prop}\label{infinitysetup}  Let 
$F \subset \widetilde {\PP^n}_{f,g}$ be the union 
of components of the total transform of the pencil such that 
for generic $t \in \PP^1$ the proper preimage 
of $t$ has non-empty intersection with $F$. Then:
\begin{enumerate}
\item The variety $F$ 
is always non-empty, and its irreducible components map surjectively onto $\PP^1$. There is a finite subset $D \subset \PP^1$  
such that 
$\widetilde \pi_{f,g}$ is a locally trivial topological fibration over $\PP^1\setminus D$ 
and, for any $t \in \PP^1 \setminus D$ , the fiber $\widetilde \pi_{f,g}^{-1}(t)$ 
is transversal to $F$.
\item The restriction of  
$\widetilde \pi_{f,g}$ to $\widetilde \pi^{-1}_{f,g}(\PP^1\setminus D) \setminus (F
\cap \widetilde \pi^{-1}_{f,g}(\PP^1\setminus D))$ is a locally 
trivial topological fibration with fiber homeomorphic to 
$\widetilde \pi_{f,g}^{-1}(t)\setminus (F \cap \widetilde \pi_{f,g}^{-1}(t))$.
\item Let $S \subset \PP^1$ be a sufficiently small disk in $\PP^1$ centered at 
$[0:1]$ (resp. at $[1: 0]$) such that 
$S \cap D=\emptyset$, and let $S^*$ be the disc $S$ punctured
at its center. 
Then $\widetilde 
\pi_{f,g}^{-1}(S^*) \setminus (F \cap \widetilde \pi_{f,g}^{-1}(S^*))$ is homeomorphic to a small punctured regular neighbourhood 
of $E_{0} \setminus (E_0 \cap F)$ (resp. ${E_{\infty}} \setminus (E_{\infty} \cap F)$) in $\widetilde {\PP^n}_{f,g} \setminus F$. 
\item Let  $c \in \CC^n \subset \PP^n$ be such that $f(c)=g(c)=0$, i.e., $c$ is an  indeterminacy point of the rational map $\pi_{f,g}$ outside the hyperplane at infinity. For sufficiently 
small $\epsilon$, let $B_\epsilon$ be a ball of radius $\epsilon$ about ${c}$ (so that the boundary of $B_{\epsilon'}$ is transversal to both $\{f=0\}$, $\{g=0\}$ and their intersection, for all $\epsilon' < \epsilon$). Finally, for $\delta << \epsilon$, let  $S^*_\delta \subset S^*$ be 
 a punctured disk, where $S$ is like in (3). Then,  the restriction of 
the map $\widetilde \pi_{f,g}$ from (2) to $\phi^{-1}(B_\epsilon) \cap \widetilde \pi_{f,g}^{-1} (S^*_\delta)$
is a locally trivial topological fibration over $S^*_\delta$. 
\end{enumerate}
\end{prop}

Using this set up we can now make the following definition.

\begin{dfn}\label{defratmil}
In the notations of Proposition \ref{infinitysetup},
\begin{enumerate}
\item the 
 Milnor fiber $M_{f,g, 0}$ (resp. $M_{f,g,\infty}$) for the value $0$ (resp. for the value $\infty$) of a rational function $f \over g$  is 
the manifold $\widetilde \pi_{f,g}^{-1}(t) \setminus (F \cap 
\widetilde \pi_{f,g}^{-1}(t))$
for any $t \in \PP^1$ closed enough to $[0:1]$ (resp. to $[1:0]$). The monodromy of this Milnor fiber is the monodromy map 
of the locally trivial fibration from Proposition \ref{infinitysetup} (2). We denote the monodromy of this fibration by $T_{f,g,0}$ (resp. by $T_{f,g,\infty}$).
\item the Milnor fiber $M_{f,g,c, 0}$  (resp. $M_{f,g,c, \infty}$) of a {\it germ} of rational 
function at an indeterminacy point $c$ for the value $0$ (resp. value $\infty$) 
is a generic fiber of the fibration from Proposition \ref{infinitysetup} (4).
We denote the monodromy of this fibration by $T_{f,g,c, 0}$ (resp. $T_{f,g,c, \infty}$).
\end{enumerate}
\end{dfn}

We shall refer to the composition $$\nabla: 
\pi_1\left(\widetilde \pi^{-1}_{f,g}(S^*)\setminus (F
\cap \widetilde\pi^{-1}_{f,g}(S^*)) \right ) \rightarrow \pi_1(S^*)=\ZZ$$
as the holonomy map of the punctured neighborhood  of 
${E_{0}\setminus F}$ (resp. ${E_{\infty}\setminus F}$) as in Proposition \ref{infinitysetup}(3).

\smallskip

\begin{remark} 
\begin{enumerate}
\item Generalizations of the notion of Milnor fiber in the context of rational functions were initiated by Gusein-Zade, Luengo and Melle-Hernandez \cite{GLM1, GLM2}, but see also \cite{BPS,T}.
\item Recall that given $f \in \CC[x_1,\cdots,x_n]$, the  
{\it Milnor fiber of $f$ at infinty} is defined as $M_f=f^{-1}(t)$ where $\vert t \vert >>0$.
Its topological type is independent of $t$, provided $\vert t \vert$ 
is sufficiently large. Moreover, its cohomology $H^{n-1}(M_f,\ZZ)$ 
is endowed with the monodromy operator induced by the trivialization 
of the bundle 
$\psi^*(M_{\vert t \vert =a})$, where 
$M_{\vert t \vert =a}$ is the preimage under $f$ of the circle 
$S_a=\{t \in \CC | \vert t \vert=a, 
a \in \RR \}$ and 
$\psi: [0,1] \rightarrow S_a$ is given by $s \mapsto a e^{2 \pi i s}$
(cf. \cite{Laffine, LS, T1}).
This notion coincides with  $M_{f,1, \infty}$ in Definition \ref{defratmil} (1).
\end{enumerate}
\end{remark}

In the above notations, we can now introduce motives associated to such topological objects (compare with \cite{R2}).

\begin{dfn} The {\it motivic Milnor fiber for the value zero (resp. for the value infinity) of a rational function} $f \over g$ is the class  
$S^A_{\widetilde \PP^n_{f,g},E_{0},\nabla} \in K_0({\rm Var}^{\hat \mu}_{\mathbb{C}})$ (resp. the class $S^A_{\widetilde \PP^n_{f,g},E_{\infty},\nabla} \in K_0({\rm Var}^{\hat \mu}_{\mathbb{C}})$),  in the sense 
of Definition \ref{maindef}, with $A$ indexing the collection of components  
of $E_0$ (resp. $E_{\infty}$) not contained in $F$.
\end{dfn}

\begin{dfn} The {\it motivic Milnor fiber of a germ of rational function} 
$f \over g$ at an indeterminacy point $c$, with $f(c)=g(c)=0$, for the value zero (resp. for the value infinity)  
is the 
class  
$S^{A(c)}_{\widetilde \PP^n_{f,g},E_{0},\nabla} \in K_0({\rm Var}^{\hat \mu}_{\mathbb{C}})$ (resp. 
the class $S^{A(c)}_{\widetilde \PP^n_{f,g},E_{\infty},\nabla} \in K_0({\rm Var}^{\hat \mu}_{\bf C})$),  
in the sense 
of Definition \ref{maindef}, with $A(c)$ indexing the collection of components 
 of $E_0$ (resp. $E_{\infty}$) not contained in $F$ and that map to the value $c$ under the resolution $\phi$ of the rational map $\pi_{f,g}$.
\end{dfn}

\begin{remark}  If $g=1$ one obtains a notion of {\it motivic Milnor fiber of $f$ at infinity}, compare for example with work by Matsui-Takeuchi \cite{MT} and Raibaut \cite{R1}. Another definition of motivic Milnor fibers for rational functions has been given by Raibaut in \cite{R2}. 
\end{remark}

Finally, as in the case of Milnor fibers of germs of polynomials, motivic Milnor fibers of rational functions have Betti realizations and there are generalizations of Corollary \ref{bettireapol} in this setting.

\begin{cor} The Betti realization of the motivic Milnor fiber 
of a rational function $f/g$ for the value zero (resp. for the value infinity) coincides  with the Betti invariant of the monodromy $T_{f,g,0}$ (resp. of the monodromy $T_{f,g,\infty}$), i.e.
$$\chi_b(S^A_{\widetilde \PP^n_{f,g},E_{\bullet},\nabla}) =  \sum_i (-1)^i[H^i_c(M_{f,g,\bullet}), T_{M_{f,g,\bullet}}^*],$$
where $\bullet$ stands for $0$ (resp. $\infty$).
\end{cor}

\begin{cor} The Betti realization of the motivic Milnor fiber 
for the value zero (resp. for the value) infinity of a {\it germ} of a rational function $f/g$ at an indeterminacy point $c$ coincides with the Betti 
invariant of the monodromy $T_{f,g,c,0}$ (resp. of the monodromy $T_{f,g,c,\infty}$), i.e.
$$\chi_b(S^{A(c)}_{\widetilde \PP^n_{f,g},E_{\bullet},\nabla}) = \sum_i (-1)^i[H^i_c(M_{f,g,c,\bullet}), T_{M_{f,g,c,\bullet}}^*],$$
where $\bullet$ stands for $0$ (resp. $\infty$).
\end{cor}

\begin{remark} $M_{f,g,0}$ and $M_{f,g,\infty}$ are members of a family 
of complex (in fact, quasi-projective) manifolds. 
This can be used to associate a limit mixed Hodge 
structure, cf. \cite{zucker}, whose motive is the Hodge realization of the above motivic Milnor fibers.
\end{remark}

\end{document}